\documentclass[11pt]{amsart}
\usepackage{amsmath}
\usepackage{amssymb}
\usepackage{amscd}

\def\NZQ{\mathbb}               
\def\NN{{\NZQ N}}

\def\ZZ{{\NZQ Z}}
\def\RR{{\NZQ R}}

\def\PP{{\NZQ P}}

\newtheorem{Theorem}{Theorem}[section]
\newtheorem{Lemma}[Theorem]{Lemma}
\newtheorem{Corollary}[Theorem]{Corollary}
\newtheorem{Proposition}[Theorem]{Proposition}
\newtheorem{Remark}[Theorem]{Remark}

\newtheorem{Example}[Theorem]{Example}

\newtheorem{Definition}[Theorem]{Definition}

\let\epsilon\varepsilon
\let\phi=\varphi
\let\kappa=\varkappa

\begin{document}
\title{Multiplicities of graded families of  linear series and ideals}
\author{Steven Dale Cutkosky }
\thanks{Partially supported by NSF}

\address{Steven Dale Cutkosky, Department of Mathematics,
University of Missouri, Columbia, MO 65211, USA}
\email{cutkoskys@missouri.edu}

\begin{abstract} We give simple necessary and sufficient conditions on projective schemes over a field $k$ for asymptotic limits of the growth of all graded linear series of a fixed Kodaira-Iitaka dimension to exist. We also give  necessary and sufficient conditions on a local ring $R$ of dimension $d$ for  the limit 
$$
\lim_{i\rightarrow\infty}\frac{\ell_R(R/I_n)}{n^d}
$$
to exist whenever $\{I_n\}$ is a 
 graded family of  $m_R$-primary ideals.
  We give a number of applications.
\end{abstract}

\maketitle

\section{Introduction} Recently, extending results in a series of papers  
(including Fujita \cite{Fuj}, Lazarseld \cite{La}, Takagi \cite{T},  Okounkov  \cite{Ok} and  Lazarsfeld and Mustata \cite{LM}),    Kaveh and Khovanskii \cite{KK} have proven  that if $L$ is
a graded linear series with Kodaira-Iitaka dimension $\kappa(L)\ge 0$    on a  projective variety $X$ over an algebraically closed field $k$,
 then there exists a positive integer $r$ such that 
\begin{equation}\label{I10}
\lim_{n\rightarrow \infty}\frac{\dim_k L_{a+nr}}{n^{\kappa(L)}}
\end{equation}
exists for any fixed $a\in \NN$.

We extend this result in \cite{C2} to hold for reduced projective schemes over a perfect field $k$, and give  examples of graded linear series on nonreduced projective schemes for which the above limits do not exist, along any arithmetic sequence. 

In this paper, we extend this result to hold over arbitrary ground fields $k$ (Theorem \ref{Theorem18}), and in Theorem \ref{TheoremN2}, give very simple necessary and sufficient conditions for a projective scheme over a field $k$ to have the property that for any given $\alpha\ge 0$, such limits exist for any graded linear system with $\kappa(L)\ge \alpha$. At the end of this introduction, we indicate the method used in the proof of Theorem \ref{Theorem18}. In \cite{C2}, we use base extension to reduce to the case when $k$ is perfect. Limits generally do not exist on nonreduced schemes, and the base extension of
a reduced scheme by a nonseparable field extension may no longer be reduced. Thus a different method is required to handle base fields which are not perfect.

We prove the following theorem about graded families of $m_R$-primary ideals. If $R$ is an excellent local ring, then Theorem \ref{TheoremI20} holds with the simpler condition that the nilradical of $R$ has dimension less that $d$ (Corollary \ref{Cornr20}).

\begin{Theorem}(Theorem \ref{Theorem4})\label{TheoremI20}  Suppose that $R$ is a local ring of dimension $d$, and  $N$ is the nilradical of the $m_R$-adic completion $\hat R$ of $R$.  Then   the limit 
\begin{equation}\label{I5}
\lim_{i\rightarrow\infty}\frac{\ell_R(R/I_n)}{n^d}
\end{equation}
exists for any graded family $\{I_i\}$ of $m_R$-primary ideals, if and only if $\dim N<d$.
\end{Theorem}

In our paper \cite{C1}, we prove that such limits exist for graded families of $m_R$-primary ideals, with the restriction that $R$ 
is an analytically unramified ($N=0$) equicharacteristic  local ring with perfect residue field (Theorem 5.8 \cite{C1}). In this paper we extend this result (in Theorem \ref{Theorem2}) to prove that the limit (\ref{I5}) exists for all graded families of $m_R$-primary ideals  in a local ring $R$ satisfying the assumptions of Theorem \ref{Theorem4}, establishing  sufficiency in Theorem \ref{Theorem4}. At the end of this introduction, we discuss the method of this proof, and what is required to make this extension.

In Example 7.1 \cite{C2},
we give an example of a graded family of $m_R$-primary ideals in a nonreduced local ring $R$ for which the above limit does not exist. 
Hailong Dao and   Ilya Smirnov have communicated to me that they have  also found this example, and have extended it to prove that if the nilradical of 
$R$ has dimension $d$, then there exists a graded family of $m_R$-primary ideals such that the  limit (\ref{I5}) does not exist (Theorem \ref{Theorem3} of this paper). Since a graded family of $m_R$-primary ideals on the completion of a ring lifts to the ring, necessity in Theorem \ref{Theorem4} follows 
from this result.

In the final Section \ref{SecApp} of this paper, we give some applications of this result and the method used in proving it, which generalize some of the applications in \cite{C1}. We extend the theorems to remove the requirement that the local ring be equicharacteristic with perfect residue field, to hold
on arbitrary analytically unramified  local rings.

We prove some volume = multiplicity formulas for graded families of $m_R$-primary ideals in  analytically unramified local rings in Theorems \ref{Theorem15} 
- \ref{Theorem13}. 
Theorem \ref{Theorem15} is proven for valuation ideals associated to an Abhyankar valuation in a regular local ring which is essentially of finite type over a field by Ein, Lazarsfeld and Smith in  \cite{ELS}, for general families of $m_R$-primary ideals when $R$ is a regular local ring containing a field by Mustata in \cite{Mus} and when $R$ is a local domain which is essentially of finite type over an algebraically closed field $k$ with $R/m_R=k$ by Lazarsfeld and Ein in Theorem 3.8 \cite{LM}. 
It is proven when $R$ is regular or $R$ is analytically unramified with perfect residue field in Theorem 6.5 \cite{C1}.

In Theorem \ref{Theorem100}, we prove such a formula for graded linear series on a proper variety over a field, extending the formula for projective varieties over an algebraically closed field proven by Koveh and Khovanskii (Theorem 3.3 \cite{KK}).

We also give, in Theorems \ref{Theorem14} - Theorem \ref{Theorem15}, some formulas showing that limits of the epsilon multiplicity type exist in analytically unramified local rings. We extend results of \cite{C1}, where it is assumed that $R$ is equicharacteristic with perfect residue field.
Epsilon multiplicity is defined as a limsup by Ulrich and Validashti in \cite{UV}.  We also extend a proof in \cite{C1} of an asymptotic formula on multiplicities proposed by Herzog, Puthenpurakal and Verma \cite{HPV}, to hold on analytically unramified local rings.

\subsection{Limits of graded families of linear series} Let $k$ be a field.

Before discussing the existence of limits, it is necessary to consider the  courser measure of the rate of growth which is possible for a graded linear series on a proper scheme over a field.
On a projective variety, there is a clear correspondence between the growth rate of a graded linear series $L=\bigoplus_{n\ge 0}L_n$ and its Kodaira-Iitaka dimension $\kappa(L)$. We have the following statements:

If $X$ is a proper $k$-variety and $L$ is a graded linear series on $X$, then
\begin{equation}\label{I3}
\dim_kL_n=0\mbox{ for all $n>0$ if and only if }\kappa(L)=-\infty.
\end{equation}

Suppose that $X$ is a proper $k$-variety and $L$ is a graded linear series on $X$ with $\kappa(L)\ge 0$. Then 
there exists a constant $\beta>0$ such that 
\begin{equation}\label{I1}
\dim_kL_n<\beta n^{\kappa(L)}
\end{equation}
for all $n$.
Further, there exists a positive number $m$ and a constant $\alpha>0$ such that 
\begin{equation}\label{I2}
\alpha n^{\kappa(L)} < \dim_kL_{mn}
\end{equation}
for all $n\gg 0$.

The statements (\ref{I3}) - (\ref{I2}) are classical for section rings of line bundles on a projective variety; they can be found  in Chapter 10 of \cite{I} and Chapter 2 of \cite{La}. The statements  hold  for linear series on a projective variety over an algebraically closed field (this is immediate from  \cite{KK}). 
In fact, the statements (\ref{I3})- (\ref{I2}) hold for linear series on a reduced proper scheme $X$ over a field $k$, as follows from Lemma \ref{LemmaKI} and 1) of Theorem \ref{Theorem8}.

From (\ref{I1}) and (\ref{I2}) we have that both $\liminf_{n\rightarrow\infty} \frac{\dim_k L_n}{n^{\kappa(L)}}$ and $\limsup_{n\rightarrow\infty} \frac{\dim_k L_n}{n^{\kappa(L)}}$ exist for a graded linear series $L$ on a proper variety. The remarkable fact is that when interpreted appropriately, they actually exist as limits on a proper variety. 

Suppose that $L$ is a graded linear series on a proper variety $X$ over a field $k$. 
The {\it index} $m=m(L)$ of $L$ is defined as the index of groups
$$
m=[\ZZ:G]
$$
where $G$ is the subgroup of $\ZZ$ generated by $\{n\mid L_n\ne 0\}$.

\begin{Theorem}\label{TheoremI1}  Suppose that $X$ is a $d$-dimensional proper variety over a field $k$, and $L$ is a graded linear series on $X$ with Kodaira-Iitaka dimension $\kappa=\kappa(L)\ge 0$. Let $m=m(L)$ be the index of $L$.  Then  
$$
\lim_{n\rightarrow \infty}\frac{\dim_k L_{nm}}{n^{\kappa}}
$$
exists. 
\end{Theorem}

In particular, from the definition of the index, we have that the limit
$$
\lim_{n\rightarrow \infty}\frac{\dim_k L_{n}}{{n}^{\kappa}}
$$
exists, whenever $n$ is constrained to lie in an arithmetic sequence $a+bm$ ($m=m(L)$ and $a$ an arbitrary but fixed constant), as $\dim_kL_n=0$ if $m\not\,\mid n$.

An example of a big line bundle where the limit in Theorem \ref{Theorem5} is an irrational number is given in Example 4 of Section 7 of the author's paper \cite{CS} with Srinivas.

Theorem \ref{TheoremI1} is proven for big line bundles on a nonsingular variety over an algebraically closed field of characterstic zero by Lazarsfeld (Example 11.4.7 \cite{La}) using Fujita approximation (Fujita, \cite{Fuj}). This result is extended by Takagi using De Jong's theory of alterations \cite{DJ} to hold on nonsingular varieties over algebraically fields of all characteristics $p\ge 0$.
Theorem \ref{TheoremI1} has been proven by  Okounkov  \cite{Ok} for section rings of ample line bundles,  Lazarsfeld and Mustata \cite{LM} for section rings of big line bundles, and for graded linear series by Kaveh and Khovanskii \cite{KK}, using an ingeneous method to reduce  to a problem of
counting points in an integral semigroup. All of these proofs require the assumption that {\it $k$ is  algebraically closed}. The theorem has been proven by the author when $k$ is a perfect field in \cite{C2}. In this paper we establish Theorem \ref{TheoremI1} over an arbitrary ground field (in Theorem \ref{Theorem5}).

The statement of Theorem \ref{TheoremI1} generalizes very nicely to reduced proper $k$-schemes, as we establish in Theorem \ref{Theorem18}. This theorem is proven for reduced projective $k$-schemes over a perfect field $k$ in \cite{C2}.

\begin{Theorem}(Theorem \ref{Theorem18}) Suppose that $X$ is a reduced proper scheme over a  field $k$. 
Let $L$ be a graded linear series on $X$ with Kodaira-Iitaka dimension  $\kappa=\kappa(L)\ge 0$.  
 Then there exists a positive integer $r$ such that 
$$
\lim_{n\rightarrow \infty}\frac{\dim_k L_{a+nr}}{n^{\kappa}}
$$
exists for any fixed $a\in \NN$.
\end{Theorem}
The theorem says that 
$$
\lim_{n\rightarrow \infty}\frac{\dim_k L_{n}}{n^{\kappa}}
$$
exists if $n$ is constrained to lie in an arithmetic sequence $a+br$ with $r$ as above, and for some fixed $a$. The conclusions of the theorem are a little weaker than the conclusions of Theorem \ref{TheoremI1} for  varieties. In particular, the index $m(L)$ has little relevance on reduced but nonirreducible schemes (as shown by the example after Theorem \ref{Theorem8} and Example 5.5 \cite{C2}).

Now we turn to the case of  nonreduced proper schemes. We begin by discussing the relationship between growth rates and the Kodaira-Iitaka dimension of graded linear series, which is much more subtle on nonreduced schemes.

A  bound (\ref{I2}) holds for any  graded linear series $L$ on a  proper $k$-scheme $X$ (Lemma \ref{LemmaKI}). 
 However, neither (\ref{I3}) nor (\ref{I1}) are true on nonreduced proper $k$-schemes. The statements must be modified
in this case to include information about the nilradical of $X$.  An upper  bound which holds on all proper $k$-schemes is obtained in Theorem \ref{Theorem8}.
 Suppose that $X$ is a  proper scheme over a field $k$. Let $\mathcal N_X$ be the nilradical of $X$. Suppose that $L$ is a graded linear series on $X$. Then
 there exists a positive constant $\gamma$ such that $\dim_kL_n<\gamma n^e$ where 
\begin{equation}\label{I4}
e=\max\{\kappa(L),\dim \mathcal N_X\}.
\end{equation}
This is the best bound possible. It is shown in Theorem \ref{TheoremN20} that if $X$ is a nonreduced projective $k$-scheme, then for any $s\in \NN\cup\{-\infty\}$ with $s\le \dim \mathcal N_X$, there exists a graded linear series $L$ on $X$ with $\kappa(L)=s$ and a constant $\alpha>0$ such that
$$
\alpha n^{\dim \mathcal N_X} < \dim_kL_n
$$
for all $n>0$.

From the bound (\ref{I4}), we see that the growth bounds (\ref{I1}) and (\ref{I2}) do in fact hold for graded linear series $L$ on a proper $k$-scheme $X$ whenever $\kappa(L)\ge \dim \mathcal N_X$.

In Example 6.3 \cite{C2} we give an example of a graded linear series $L$ of maximal Kodaira-Iitaka dimension $\kappa(L)=d=\dim X$ such that the limit   
$$
\lim_{n\rightarrow \infty}\frac{\dim_k L_{n}}{n^{d}}
$$
does not exist, even when $n$ is constrained to lie in {\it any} arithmetic sequence. However, we prove in Theorem \ref{Theorem8} of this paper that if 
$L$ is a graded linear series on a proper scheme $X$ over a field $k$, with $\kappa(L)>\dim \mathcal N_X$, then 
there exists a positive integer $r$ such that 
$$
\lim_{n\rightarrow \infty}\frac{\dim_k L_{a+nr}}{n^{\kappa}}
$$
exists for any fixed $a\in \NN$.

This is the strongest statement on limits that is true. In fact, the existence of all such limits characterizes the dimension of the nilradical, at least on projective schemes.
We show this in the following theorem.

\begin{Theorem}(Theorem \ref{TheoremN2})
Suppose that $X$ is a $d$-dimensional projective scheme over a field $k$ with $d>0$. Let $\mathcal N_X$ be the nilradical of $X$. Let $\alpha\in\NN$. Then the following are equivalent:
\begin{enumerate}
\item[1)] For every graded linear series $L$ on $X$ with $\alpha\le\kappa(L)$, there exists a positive integer $r$ such that 
$$
\lim_{n\rightarrow\infty}\frac{\dim_kL_{a+nr}}{n^{\kappa(L)}}
$$
exists for every positive integer $a$.
\item[2)] For every graded linear series $L$ on $X$ with $\alpha\le \kappa(L)$, there exists an arithmetic sequence $a+nr$ (for fixed $r$ and $a$ depending on $L$) such that 
$$
\lim_{n\rightarrow\infty}\frac{\dim_kL_{a+nr}}{n^{\kappa(L)}}
$$
exists.
\item[3)] The nilradical $\mathcal N_X$ of $X$ satisfies $\dim \mathcal N_X<\alpha$.
\end{enumerate}
\end{Theorem}

If $X$ is a $k$-scheme of  dimension $d=0$ which is not irreducible, then the conclusions of Theorem \ref{TheoremN2} are true for $X$. This follows from Section \ref{SecZero}. However, 2) implies 3)  does not hold if $X$ is irreducible of dimension 0. In fact (Proposition \ref{TheoremN3}) if $X$ is irreducible of dimension 0,
and $L$ is a graded linear series on $X$ with $\kappa(L)=0$,  then there exists a positive integer $r$ such that the limit 
$\lim_{n\rightarrow \infty}\dim_kL_{a+nr}$ exists for every positive integer $a$.

\subsection{The method} Now we present a method to reduce problems on computing limits of lengths of graded families of modules to problems
on integral semigroups, so that the techniques for counting points in integral semigroups developed by Khovanskii \cite{Kh} and Kaveh and Khovanskii \cite{KK}
are applicable. The method  appears in work of Okounkov \cite{Ok} and has been refined by Lazarsfeld and Mustata \cite{LM} and Kaveh and Khovanskii \cite{KK}, to apply to graded linear series on projective varieties over an algebraically closed field $k$. In these papers, (using the notation below) $F$ is an algebraic function field of a $d$-dimensional variety, $\nu$ is a rank $d$ valuation, $R=k$ and the residue field $V_{\nu}/m_R=k$ (this condition on the valuation is  called having ``one dimensional leaves'' in \cite{KK}). The $M_i$ are a graded linear series, so that $\dim_k M_i<\infty$ for all $i$. 

The method is refined in \cite{C1} to include the case when $R$ is an analytically irreducible local ring, but still with the restriction that $[V_{\nu}/m_{\nu}:k]=1$ (with $k=R/m_R$). Now this is a major restriction. In fact, even in the context of graded linear series on a projective variety,
there exists such a valuation with $V_{\nu}/m_{\nu}$ equal to the ground field $k$ if and only if there exists a 
birational morphism $X'\rightarrow X$ and a nonsingular (regular) $k$-rational point $Q'\in X'$ (Proposition 4.2 \cite{C2}). 

In \cite{C1}, the restriction that $R$ be analytically irreducible is necessary to ensure that the topologies defined by the $m_R$-adic topology and the
filtration defined by the ideals $\{R\cap K_{\lambda}\}$, where $K_{\lambda}$ is the ideal in $V_{\nu}$ of elements of value $\ge \lambda$, are linearly equivalent. This is necessary to calculate the $\beta$ below when approximating lengths of quotients by graded families of $m_R$-primary ideals.

In this paper we generalize the method to work whenever $[V_{\nu}/m_{\nu}:k]<\infty$. This is sufficiently general to allow for very wide applicability. 
For instance, every projective variety over a field $k$ has a closed point $Q$ such that $\mathcal O_{X,Q}$ is regular, and certainly 
$$
[k(Q):k]<\infty.
$$

In our papers \cite{C1} and \cite{C2}, we use  base extension to reduce to the case where $k=V_{\nu}/m_{\nu}$. As discussed earlier, limits do not always exist when there are nilpotents.
Thus base change cannot be used in the case when the extension $V_{\nu}/m_{\nu}$ is not separable over $k$. In particular, even for the case of graded linear series,  this generalization is required to obtain results over
arbitrary (possibly nonperfect) ground fields.

Now we present an outline of the method. It is presented in detail in the body of the paper where it is  applied.
Suppose that $F$ is a field, and $\nu$ is a valuation of $F$ with valuation ring $V_{\nu}$ and value group $\Gamma_{\nu}=\ZZ^d$ with some total order
(in the case of graded families of $m_R$-primary ideals we take $\Gamma_{\nu}$ to be an ordered subgroup of the real numbers, and in the case of graded linear series on a projective variety  we take $\Gamma_{\nu}$ to be $\ZZ^d$ with the lex order).
Suppose that $R\subset V_{\nu}$ is a subring such that $V_{\nu}$ dominates $R$, and $[V_{\nu}/m_{\nu}:k]<\infty$, where $k=R/m_R$.
Suppose that $M_i\subset V_{\nu}$ are $R$-modules such that $M_iM_j\subset M_{i+j}$ for all $i,j$.
For all $\vec n\in \ZZ^d$, define ideals in $V_{\nu}$ by
$$
K_{\vec n} = \{f\in V_{\nu}\mid \nu(f)\ge \vec n\},
$$
$$
K_{\vec n}^+ = \{f\in V_{\nu}\mid \nu(f)> \vec n\}.
$$
Suppose that $\beta \in\Gamma_{\nu}$ is such that the length
$$
\ell_R(M_i/M_i\cap K_{\beta i})<\infty
$$
for all $i$. For $1\le i\le [V_{\nu}/m_{\nu}:k]$, define subsets of $\ZZ^{d+1}$ by
$$
S(M)^{(t)} = \{ (\overline n,i)\in \ZZ^{d+1}\mid \dim_k M_i\cap K_{\overline n}/M_i\cap K_{\overline n}^+\ge t\}.
$$
The $S(M)^{(t)}$ are in fact subsemigroups of $\ZZ^{d+1}$. Letting $S(M)^{(t)}_i=S(M)^{(t)}\cap (\ZZ^d\times\{i\})$, we compute the length
$$
\ell_R(M_i/M_i\cap K_{\beta i})=\sum_{t=1}^{[V_{\nu}/m_{\nu}:k]}\#(S(M)_i^{(t)}).
$$
This converts the problem of computing lengths into a problem of computing integral points of integral semigroups.

\section{notation and conventions}\label{SecNot} $m_R$ will denote the maximal ideal of a local ring $R$. $Q(R)$ will denote the quotient field of a domain $R$.
$\ell_R(N)$ will denote the length of an $R$-module $N$.  $\ZZ_+$ denotes the positive integers and $\NN$ the nonnegative integers. 
Suppose that $x\in \RR$. $\lceil x \rceil$ is the smallest integer $n$ such  that $x\le n$. $\lfloor x \rfloor$ is the largest integer $n$ such that $n\le x$. 

We recall some notation on multiplicity from Chapter VIII, Section 10 of \cite{ZS2}, Section V-2 \cite{Se} and Section 4.6 \cite{BH}.
Suppose that $(R,m_R)$ is a (Noetherian) local ring,  $N$ is a finitely generated $R$-module with $r=\dim N$ and $a$ is an ideal of definition of $R$. Then
$$
e_a(N)=\lim_{k\rightarrow \infty}\frac{\ell_R(N/a^kN)}{k^r/r!}.
$$
We write $e(a)=e_a(R)$.

If $s\ge r=\dim N$, then we define 
$$
e_s(a,N)=\left\{
\begin{array}{ll}
 e_a(N)&\mbox{ if }\dim N=s\\
 0&\mbox{ if } \dim N<s.
 \end{array}\right.
 $$

A local ring is analytically unramified if its completion is reduced. In particular, a reduced excellent local ring is 
analytically unramified.

 We will denote the maximal ideal of a local ring $R$ by $m_R$. If $\nu$ is a valuation of a field $K$, then we will write $V_{\nu}$ for the valuation ring of $\nu$, and $m_{\nu}$ for the maximal ideal of $V_{\nu}$. We will write $\Gamma_{\nu}$ for the value group of $\nu$. If $A$ and $B$ are local rings, we will say that $B$ dominates $A$ if $A\subset B$ and $m_B\cap A=m_A$.
 
 We  use the notation of Hartshorne \cite{H}. For instance, a variety is required to be integral.
If $\mathcal F$ is a coherent sheaf on a Noetherian scheme, then $\dim \mathcal F$ will denote the dimension of the support of $\mathcal F$, with  $\dim \mathcal F=-\infty$ if $\mathcal F=0$.

Suppose that $X$ is a scheme. The nilradical of $X$ is the ideal sheaf $\mathcal N_X$ on $X$ which is the kernel of the natural surjection
$\mathcal O_{X}\rightarrow \mathcal O_{X_{\rm red}}$ where $X_{\rm red}$ is the reduced scheme associated to $X$. $(\mathcal N_X)_\eta$ is the nilradical of the local ring $\mathcal O_{X,\eta}$ for all $\eta\in X$.

\section{Cones associated to semigroups}\label{SecCone}

In this section, we summarize some results of Okounkov \cite{Ok}, Lazarsfeld and Mustata \cite{LM} and Koveh and Khovanskii \cite{KK}. 

Suppose that $S$ is a subsemigroup of $\ZZ^{d}\times \NN$ which is not contained in $\ZZ^d\times\{0\}$. Let $L(S)$ be the subspace of $\RR^{d+1}$ which is generated by $S$, and let $M(S)=L(S)\cap(\RR^d\times\RR_{\ge 0})$. 

Let $\mbox{Con}(S)\subset L(S)$ be the closed convex cone which is the closure of  the set of all linear combinations $\sum \lambda_is_i$ with $s_i\in S$ and $\lambda_i\ge 0$.

$S$ is called {\it strongly nonnegative} (Section 1.4 \cite{KK}) if $\mbox{Cone}(S)$  intersects $\partial M(S)$ only at the origin (this is equivalent to being strongly admissible (Definition 1.9 \cite{KK}) since with our assumptions, $\mbox{Cone}(S)$ is contained in $\RR^d\times\RR_{\ge 0}$,  so the ridge of of $S$ must be contained in $\partial M(S)$). In particular, a subsemigroup of a strongly negative semigroup is itself strongly negative.

We now introduce some notation from \cite{KK}. Let 
\vskip .1truein

$S_k=S\cap (\RR^d\times\{k\})$.

$\Delta(S)=\mbox{Con}(S)\cap (\RR^{d}\times\{1\})$ (the Newton-Okounkov body of $S$).

$q(S)=\dim \partial M(S)$.

$G(S)$ be the subgroup of $\ZZ^{d+1}$ generated by $S$.

$m(S)=[\ZZ:\pi(G(S))]$
  where $\pi:\RR^{d+1}\rightarrow \RR$ be projection onto the last factor.

$\mbox{ind}(S)= [\partial M(S)_{\ZZ}:G(S)\cap \partial M(S)_{\ZZ}]$
where 

$\partial M(S)_{\ZZ}:=\partial M(S)\cap \ZZ^{d+1}= M(S)\cap (\ZZ^d\times\{0\})$.

${\rm vol}_{q(S)}(\Delta(S))$ is the integral volume of $\Delta(S)$. This volume is computed using the translation of the integral measure on $\partial M(S)$.
\vskip .2truein

$S$ is strongly negative if and only if $\Delta(S)$ is a compact set. If $S$ is strongly negative, then the dimension of $\Delta(S)$ is $q(S)$.
\vskip .1truein

\begin{Theorem}\label{ConeTheorem3}(Kaveh and Khovanskii) Suppose that $S$ is strongly nonnegative.  Then 
$$
\lim_{k\rightarrow \infty}\frac{\#S_{m(S)k}}{k^{q(S)}}=\frac{{\rm vol}_{q(S)}(\Delta(S))}{{\rm ind}(S)}.
$$
\end{Theorem}

This is proven in  Corollary 1.16 \cite{KK}. 

With our assumptions, we have that $S_n=\emptyset$ if $m(S)\not\,\mid n$ and  the limit is positive, since
${\rm vol}_{q(S)}(\Delta(S))>0$.

\begin{Theorem}\label{ConeTheorem1}(Okounkov,  Section 3 \cite{Ok}, Lazarsfeld and Mustata, Proposition 2.1 \cite{LM}) Suppose that a subsemigroup $S$ of  $\ZZ^d\times\ZZ_{\ge 0}$  satisfies the following two conditions:
\begin{equation}\label{Cone2}
\begin{array}{l}
\mbox{There exist finitely many vectors $(v_i,1)$ spanning a semigroup $B\subset\NN^{d+1}$}\\
\mbox{such that $S\subset B$}
\end{array}
\end{equation}
and
\begin{equation}\label{Cone3}
G(S)=\ZZ^{d+1}.
\end{equation}
Then
$$
\lim_{n\rightarrow\infty} \frac{\# S_n}{n^d}={\rm vol}(\Delta(S)).
$$
\end{Theorem}

\begin{proof} 
$S$ is strongly nonnegative since $B$ is strongly nonnegative, so Theorem \ref{ConeTheorem3} holds.

$G(S)=\ZZ^{d+1}$ implies $L(S)=\RR^{d+1}$, so $M(S)=\RR^d\times\RR_{\ge 0}$, $\partial M(S)=\RR^d\times\{0\}$ and 
$q(S)=\dim \partial M(S)=d$. We thus have  $m(S)=1$ and $\mbox{ind}(S)=1$. 
\end{proof}

\begin{Theorem}\label{ConeTheorem4} Suppose that $S$ is  strongly nonnegative. Fix $\epsilon>0$. Then there is an integer $p=p_0(\epsilon)$ such that if $p\ge p_0$, then the limit
$$
\lim_{n\rightarrow\infty}\frac{\#(n*S_{pm(S)})}{n^{q(S)}p^{q(S)}}\ge \frac{{\rm vol}_{q(S)}\Delta(S)}{{\rm ind}(S)}-\epsilon
$$
exists, where
$$
n*S_{pm(S)}=\{x_1+\cdots+x_n\mid x_1,\ldots,x_n\in S_{pm(S)}\}.
$$
\end{Theorem}

\begin{proof} Let $m=m(S)$ and $q=q(S)$. Let $S^{[pm]}=\cup_{n=1}^{\infty} (n*S_{pm(S)})$ be the subsemigroup of $S$ generated by $S_{pm}$. For $p\gg 0$, we have that $L(S^{[pm]})=L(S)$ so $m(S^{[pm]})=pm$ and $q(S^{[pm]})=q$. 

Suppose that $v_1,\ldots,v_r$ generate $G(S)\cap \partial M(S)_{\ZZ}$. For $1\le i\le r$, there exist $a_i,b_i,n_i$ such that
$v_i=(a_i,n_im)-(b_i,n_im)$ with $(a_i,n_im), (b_i,n_im)\in S_{n_im}$. There exist $b>0$ and $c,c'$ such that $(c,mb) \in S$ and $(c',m(b+1))\in S$.
$bm$ divides $n_im+n_i(b-1)(b+1)m$ and
$$
v_i=[(a_i,n_im)+n_i(b-1)(c',(b+1)m)]-[(b_i,n_im)+n_i(b-1)(c',(b+1)m)],
$$
so we may assume that $b$ divides $n_i$ for all $i$. Thus $v_1,\ldots,v_r\in G(S^{[nm]})$ where $n=\max\{n_i\}$, and $v_1,\ldots,v_r\in G(S^{[pm]})$ whenever $p\ge (b-1)b+n$. Thus
\begin{equation}\label{eqnr70}
{\rm ind}(S^{[pm]})={\rm ind}(S)
\end{equation}
whenever $p\gg 0$.
We have that
\begin{equation}\label{eqnr71}
\lim_{p\rightarrow\infty} \frac{{\rm vol}_q(\Delta(S^{[pm]}))}{p^q}={\rm vol}_q(\Delta(S).
\end{equation}
By Theorem \ref{ConeTheorem3},
\begin{equation}\label{eqnr72}
\lim_{n\rightarrow \infty}\frac{\#(n*S_{pm})}{n^q}=\frac{{\rm vol}_q(\Delta(S^{[pm]}))}{{\rm ind}(S^{[pm]})}.
\end{equation}
The theorem now follows from (\ref{eqnr70}), (\ref{eqnr71}), (\ref{eqnr72}).
 \end{proof}
 
 We obtain the following result.

\begin{Theorem}\label{ConeTheorem2}(Proposition 3.1 \cite{LM}) Suppose that a subsemigroup $S$ of $\ZZ^d\times\ZZ_{\ge 0}$ satisfies  (\ref{Cone2}) and (\ref{Cone3}). Fix $\epsilon>0$. Then there is an integer $p_0= p_0(\epsilon)$
such that if $p\ge p_0$, then the limit
$$
\lim_{k\rightarrow \infty} \frac{\#(k*S_p)}{k^dp^d}\ge {\rm vol}(\Delta(S))-\epsilon
$$
exists.
\end{Theorem}

\section{Asymptotic theorems on lengths}\label{SecAsyRing}

\begin{Definition} A graded family of ideals $\{I_i\}$ in a ring $R$ is a family of ideals indexed by the natural numbers such that $I_0=R$ and $I_iI_j\subset I_{i+j}$
for all $i,j$.  If $R$ is a local ring and $I_i$ is $m_R$-primary for $i>0$, then we will say that $\{I_i\}$ is a graded family of $m_R$-primary ideals.
\end{Definition}

The following theorem is proven with the further assumptions that $R$ is equicharacteristic with perfect residue field in \cite{C1}.

\begin{Theorem}\label{Theorem1} Suppose that $R$ is an analytically irreducible local domain of dimension $d$ and  $\{I_i\}$ is a graded family of  $m_R$-primary ideals in $R$. 
Then
$$
\lim_{i\rightarrow\infty}\frac{\ell_R(R/I_i)}{i^d}
$$
exists.
\end{Theorem}

\begin{Corollary}\label{Cor33} Suppose that $R$ is an analytically irreducible local domain of dimension $d>0$ and  $\{I_i\}$ is a graded family of   ideals in $R$ such that there exists a positive number $c$ such that $m_R^c\subset I_1$.
Then
$$
\lim_{i\rightarrow\infty}\frac{\ell_R(R/I_i)}{i^d}
$$
exists.
\end{Corollary}

\begin{proof}
The assumption $m_R^c\subset I_1$ implies that either $I_n$ is $m_R$-primary for all positive $n$, or there exists
$n_0>1$ such that $I_{n_0}=R$. In the first case, the Corollary follows from Theorem \ref{Theorem1}. In the second case, $m_R^{cn_0}\subset I_n$ for all $n\ge n_0$,  so $\ell_R(R/I_i)$ is actually bounded.
\end{proof}

We now give the proof of Theorem \ref{Theorem1}.

Since $I_1$ is $m_R$-primary, there exists   $c\in \ZZ_+$ such that 
\begin{equation}\label{eq8}
m_R^c\subset I_1.
\end{equation}

 Let $\hat R$ be the $m_R$-adic completion of $R$. Since the $I_n$ are $m_R$-primary, we have that 
$R/I_n\cong \hat R/I_n\hat R$ and $\ell_R(R/I_n)=\ell_{\hat R}(\hat R/I_n\hat R)$ for all $n$.   We may thus assume that $R$ is an excellent domain.
Let $\pi:X\rightarrow \mbox{spec}(R)$ be the normalization of the blow up of $m_R$. $X$ is of finite type over $R$ since $R$ is excellent.
Since $\pi^{-1}(m_R)$ has codimension 1 in $X$ and $X$ is normal, there exists a closed point $x\in X$ such that the local ring $\mathcal O_{X,x}$ is a regular local ring. Let $S$ be this local ring. $S$ is a regular local ring which  is essentially of finite type  and birational over $R$ ($R$ and $S$ have the same  quotient field).

Let $y_1,\ldots,y_d$ be a regular system of parameters in $S$. Let $\lambda_1,\ldots,\lambda_d$ be rationally independent real numbers, such that 
\begin{equation}\label{eq9}
\lambda_i\ge 1\mbox{ for all $i$}.
\end{equation}
 We define a valuation $\nu$ on
$Q(R)$ which dominates $S$ by prescribing 
$$
\nu(y_1^{a_1}\cdots y_d^{a_d})=a_1\lambda_1+\cdots+a_d\lambda_d
$$
for $a_1,\ldots,a_d\in \ZZ_+$, and $\nu(\gamma)=0$  if $\gamma\in S$ has nonzero residue in $S/m_S$.

Let $C$ be a coefficient set of $S$. Since $S$ is a regular local ring, for $r\in \ZZ_+$ and $f\in S$, there is a unique expression 
\begin{equation}\label{eqred11}
f=\sum s_{i_1,\ldots,i_d}y_1^{i_1}\cdots y_d^{i_d}+g_r
\end{equation}
with $g_r\in m_S^r$, $s_{i_1,\ldots,i_d}\in S$ and $i_1+\cdots+i_d<r$ for all $i_1,\ldots,i_d$ appearing in the sum. Take $r$ so large that 
$r> i_1\lambda_1+\cdots+i_d\lambda_d$ for some term with $s_{i_1,\ldots,i_d}\ne 0$. Then define
\begin{equation}\label{eq61}
\nu(f)=\min\{i_1\lambda_1+\cdots+i_d\lambda_d\mid s_{i_1,\ldots,i_d}\ne 0\}.
\end{equation}
This definition is well defined, and we calculate that
$\nu(f+g)\ge \min\{\nu(f),\nu(g)\}$ and $\nu(fg)=\nu(f)+\nu(g)$ (by the uniqueness of the expansion (\ref{eqred11})) for all $0\ne f,g\in S$. Thus $\nu$ is a valuation.
 Let $V_{\nu}$ be the valuation ring of $\nu$ (in $Q(R)$). The value group $\Gamma_{\nu}$ of $V_{\nu}$ is the (nondiscrete) ordered subgroup 
$\ZZ\lambda_1+\cdots+\ZZ\lambda_d$ of $\RR$. Since there is unique monomial giving the minimum in (\ref{eq61}), we have that the residue field of $V_{\nu}$ is
$S/m_S$.

Let $k=R/m_R$ and $k'=S/m_S=V_{\nu}/m_{\nu}$. Since $S$ is essentially of finite type over $R$, we have that $[k':k]<\infty$.

For $\lambda\in \RR$, define  ideals $K_{\lambda}$ and $K_{\lambda}^+$ in $V_{\nu}$ by
$$
K_{\lambda}=\{f\in Q(R)\mid \nu(f)\ge\lambda\}
$$
and
$$
K_{\lambda}^+=\{f\in Q(R)\mid \nu(f)>\lambda\}.
$$

We follow the usual convention that $\nu(0)=\infty$ is larger than any element of $\RR$. By Lemma 4.3 \cite{C1}, we have the following formula. The assumption that $R$ is analytically irreducible is necessary for the validity of the formula.

\begin{equation}\label{eqred50}
\mbox{There exists $\alpha\in \ZZ_+$ such that $K_{\alpha n}\cap R\subset m_R^n$ for all $n\in \NN$.}
\end{equation}

 Suppose that $I\subset R$ is an ideal and $\lambda\in \Gamma_{\nu}$ is nonnegative. Then we have inclusions of $k$-vector spaces
 $$
 I\cap K_{\lambda}/I\cap K_{\lambda}^+\subset K_{\lambda}/K_{\lambda}^+.
 $$
 Since $K_{\lambda}/K_{\lambda}^+$ is isomorphic to $k'$, we conclude that 
\begin{equation}\label{red1}
\dim_k I\cap K_{\lambda}/I\cap K_{\lambda}^+\le [k':k].
\end{equation}

 Let $\beta=\alpha c\in \ZZ_+$, where $c$ is the constant of (\ref{eq8}), and $\alpha$ is the constant of (\ref{eqred50}), so that for all $i\in \ZZ_+$,
\begin{equation}\label{eq13}
K_{\beta i}\cap R= K_{\alpha c i}\cap R\subset m_R^{ic}\subset I_i.
\end{equation}

For $t\ge  1$, define 
$$
\Gamma^{(t)}=\left\{
\begin{array}{l}
(n_1,\ldots,n_d,i)\in \NN^{d+1}\mid \dim_k I_i\cap K_{n_1\lambda_1+\cdots+n_d\lambda_d}/I_i\cap K_{n_1\lambda_1+\cdots+n_d\lambda_d}^+\ge t\\
\mbox{ and }n_1+\cdots+n_d\le \beta i
\end{array}
\right\},
$$
and
$$
\hat{\Gamma}^{(t)}=
\left\{
\begin{array}{l}(n_1,\ldots,n_d,i)\in \NN^{d+1}\mid \dim_k R\cap K_{n_1\lambda_1+\cdots+n_d\lambda_d}/R\cap K_{n_1\lambda_1+\cdots+n_d\lambda_d}^+\ge t\\
\mbox{ and }n_1+\cdots+n_d\le \beta i
\end{array}
\right\}.
$$

Let $\lambda=n_1\lambda_1+\cdots+n_d\lambda_d$ be such that $n_1+\cdots+n_d\le\beta i$. Then
\begin{equation}\label{eqred40}
\dim_kK_{\lambda}\cap I_i/K_{\lambda}^+\cap I_i=\#\{t|(n_1,\ldots,n_d,i)\in \Gamma_i^{(t)}\}.
\end{equation}

\begin{Lemma}\label{Lemmared1} Suppose that $t\ge 1$, $0\ne f\in I_i$, $0\ne g\in I_j$ and
$$
\dim_kI_i\cap K_{\nu(f)}/I_i\cap K_{\nu(f)}^+\ge t.
$$
Then
\begin{equation}\label{eqred10}
\dim_k I_{i+j}\cap K_{\nu(fg)}/I_{i+j}\cap K_{\nu(fg)}^+\ge t.
\end{equation}
In particular, when nonempty,
   $\Gamma^{(t)}$ and $\hat{\Gamma}^{(t)}$ are subsemigroups of the  semigroup $\ZZ^{d+1}$.
\end{Lemma}

\begin{proof} 
 There exist $f_1,\ldots,f_t\in I_i\cap K_{\nu(f)}$ such that their classes  are linearly independent over $k$ in $I_i\cap K_{\nu(f)}/ I_i\cap K_{\nu(f)}^+$.
 We will show that the classes of $gf_1,\ldots, gf_t$ in 
$$
I_{i+j}\cap K_{\nu(fg)}/I_{i+j}\cap K_{\nu(fg)}^+
$$
 are linearly independent over $k$.

Suppose that 
$a_1,\ldots,a_t \in k$  are such that the class of $a_1gf_1+\cdots +a_tgf_t$ in $I_{i+j}\cap K_{\nu(fg)}/I_{i+j}\cap K_{\nu(fg)}^+$ is zero. Then
$\nu(a_1gf_1+\cdots +a_tgf_t)>\nu(fg)$, whence $\nu(a_1f_1+\cdots+a_tf_t)>\nu(f)$,
so $a_1f_1+\cdots +a_tf_t\in I_i\cap K_{\nu(f)}^+$. Thus $a_1=\cdots=a_t=0$, since the classes of $f_1,\ldots,f_t$ are linearly independent over $k$ in $I_i\cap K_{\nu(f)}/ I_i\cap K_{\nu(f)}^+$. 
\end{proof}

From (\ref{eq13}), and since $n_1\lambda_1+\cdots+n_d\lambda_d<\beta i$
implies $n_1+\cdots+n_d< \beta i$   by (\ref{eq9}),
we have that 

\begin{equation}\label{eq12}
\begin{array}{lll}
\ell_R(R/I_i)&=&\ell_R(R/K_{\beta i}\cap R)-\ell_R(I_i/K_{\beta i}\cap I_i)\\
&=& \dim_k\left(\bigoplus_{\lambda<\beta i}K_{\lambda}\cap R/K_{\lambda}^+\cap R\right)
-\dim_k\left(\bigoplus_{\lambda<\beta i}K_{\lambda}\cap I_i/K_{\lambda}^+\cap I_i\right)\\
&=& \left(\sum_{t=1}^{[k':k]}\# \hat{\Gamma}_i^{(t)}\right)-\left(\sum_{t=1}^{[k':k]}\# \Gamma_i^{(t)}\right),
\end{array}
\end{equation}
where $\Gamma_i^{(t)}=\Gamma^{(t)}\cap (\NN^d\times\{i\})$ and $\hat{\Gamma}_i^{(t)}=\hat{\Gamma}^{(t)}\cap (\NN^d\times\{i\})$.

For $0\ne f\in R$, define
$$
\phi(f)=(n_1,\ldots,n_d)\in \NN^d
$$
if $\nu(f)=n_1\lambda_1+\cdots+n_d\lambda_d$. We have that $\phi(fg)=\phi(f)+\phi(g)$.

\begin{Lemma}\label{Lemmared3} Suppose that $t\ge 1$ and $\Gamma^{(t)}\not\subset (0)$. Then  $\Gamma^{(t)}$ satisfies equations (\ref{Cone2}) and (\ref{Cone3}).
\end{Lemma}

\begin{proof} Let $\{e_i\}$ be the standard basis of $\ZZ^{d+1}$.
The semigroup
$$
B=\{(n_1,\ldots,n_d,i)\mid (n_1,\ldots,n_d)\in \NN^{d}\mbox{ and }n_1+\cdots+n_d\le\beta i\}
$$
is generated by $B\cap (\NN^d\times\{1\})$ and contains $\Gamma^{(t)}$, so (\ref{Cone2}) holds.

By assumption, there exists $r\ge 1$ and $0\ne h\in I_r$ such that $(\phi(h),r)\in \Gamma^{(t)}$.

There exists $0\ne u\in I_1$. 
Write $y_i=\frac{f_i}{g_i}$ with $f_i,g_i\in R$ for $1\le i\le d$.  Then $hf_i, hg_i\in I_{r}$. There exists $c'\in \ZZ_+$ such that $c'\ge c$ and $u, hf_i, hg_i\not\in m_R^{c'}$ for $1\le i\le d$. We may replace $c$ with $c'$ in (\ref{eq8}). Then
$(\phi(hf_i),r), (\phi(hg_i),r)\in \Gamma_r^{(t)}=\Gamma^{(t)}\cap(\NN^d\times\{r\})$ for $1\le i\le d$, by (\ref{eqred10}) (with $f_i,g_i\in I_0=R$)  since $hf_i$ and $hg_i$ all have values $n_1\lambda_1+\cdots+n_d\lambda_d<\beta r$, so that $n_1+\ldots+n_d<\beta r$. We have that  $\phi(y_i)=\phi(hf_i)-\phi(hg_i)=\phi(y_i)=e_i$ for $1\le i\le d$. Thus
$$
(e_i,0)=(\phi(hf_i),r)-(\phi(hg_i),r)\in G(\Gamma^{(t)})
$$
for $1\le i\le d$. 
 $(\phi(uh),r+1)\in \Gamma^{(t)}$ by (\ref{eqred10}) and since $\nu(u)\le \beta$,
so that $(\phi(u),1)\in G(\Gamma^{(t)})$, so $e_{d+1} \in G(\Gamma^{(t)})$. Thus $G(\Gamma^{(t)})=\ZZ^{d+1}$ and (\ref{Cone3}) holds.
\end{proof}

The  same argument proves the following lemma.
\begin{Lemma}\label{Lemmared4} Suppose that $t\ge 1$ and $\hat\Gamma^{(t)}\not\subset (0)$.  Then  $\hat{\Gamma}^{(t)}$ satisfies equations (\ref{Cone2}) and(\ref{Cone3}).
\end{Lemma}

By Theorem \ref{ConeTheorem1}, 
\begin{equation}\label{eq16}
\lim_{i\rightarrow\infty} \frac{\# \Gamma_i^{(t)}}{i^d}={\rm vol}(\Delta(\Gamma^{(t)}))
\end{equation}
and
\begin{equation}\label{eq17}
\lim_{i\rightarrow\infty} \frac{\# \hat{\Gamma}^{(t)}_i}{i^d}={\rm vol}(\Delta(\hat{\Gamma}^{(t)})).
\end{equation}
We obtain the conclusions of Theorem \ref{Theorem1} from equations (\ref{eq12}), (\ref{eq16}) and (\ref{eq17}).

\begin{Theorem}\label{Theorem2} Suppose that $R$ is a local ring of dimension $d$, and $\{I_i\}$ is a graded family of $m_R$-primary ideals in $R$. 
Let $N$ be the nilradical of the $m_R$-adic completion $\hat R$ of $R$, and suppose that $\dim N<d=\dim R$. Then
$$
\lim_{i\rightarrow \infty}\frac{\ell_R(R/I_i)}{i^d}
$$
exists.
\end{Theorem}

\begin{proof} Let $A=\hat R/N$. We have a short exact sequence of $\hat R$-modules
$$
0\rightarrow N/(N\cap I_i\hat R)\rightarrow \hat R/I_i\hat R\rightarrow A/I_iA\rightarrow 0.
$$
There exists a number $c$ such that $m_R^c\subset I_1$. Hence $m_R^{ci}N\subset N\cap I_i\hat R$ for all $i$, so that
$$
\ell_{\hat R}(N/N\cap I_i\hat R)\le \ell_{\hat R}(N/m_{\hat R}^{ci}N)\le \alpha i^{\dim N}
$$
for some constant $\alpha$. 
Hence

$$
\lim_{i\rightarrow \infty}\frac{\ell_R(R/I_i)}{i^d}=\lim_{i\rightarrow \infty}\frac{\ell_{\hat R}(\hat R/I_i\hat R)}{i^d}
=\lim_{i\rightarrow \infty}\frac{\ell_A(A/I_iA)}{i^d}.
$$

 Let $p_1,\ldots, p_s$ be the minimal primes of $A$, and $A_j=A/p_j$ for $1\le j\le s$.
By Lemma \ref{Lemma5} below,
$$
\lim_{i\rightarrow \infty}\frac{\ell_{A}(A/I_iA)}{i^d}=\sum_{j=1}^s\lim_{i\rightarrow \infty}\frac{\ell_{A_j}(A_j/I_iA_j)}{i^d}
$$
which exists by Theorem \ref{Theorem1}
\end{proof}

\begin{Lemma}\label{Lemma5}(Lemma 5.1 \cite{C1}) Suppose that $R$ is a $d$-dimensional reduced local ring and $\{I_n\}$ is a graded family of $m_R$-primary ideals in $R$,  Let $\{p_1,\ldots, p_s\}$ be the minimal primes of $R$, $R_i=R/p_i$, and let $S$ be the ring
$S=\bigoplus_{i=1}^sR_i$. Then there exists $\alpha\in \ZZ_+$ such that for all $n\in \ZZ_+$,
$$
|\sum_{i=1}^s\ell_{R_i}(R_i/I_nR_i)-\ell_R(R/I_n)|\le \alpha n^{d-1}.
$$
\end{Lemma}

\section{A necessary and sufficient condition for limits to exist in a local ring}

Let $i_1=2$ and $r_1=\frac{i_1}{2}$. For $j\ge 1$, inductively define $i_{j+1}$ so that $i_{j+1}$ is even and $i_{j+1}>2^ji_j$.
Let $r_{j+1}=\frac{i_{j+1}}{2}$. For $n\in \ZZ_+$, define

\begin{equation}\label{eqsigma}
\sigma(n)=\left\{\begin{array}{ll}
1&\mbox{ if } n=1\\
\frac{i_j}{2} &\mbox{ if }i_j\le n<i_{j+1}
\end{array}\right.
\end{equation}

The limit
\begin{equation}\label{eqnr1}
\lim_{n\rightarrow \infty} \frac{\sigma(n)}{n}
\end{equation}
does not exist, even when $n$ is constrained to lie in an  arithmetic sequence (this is shown in Lemmas 6.1 and 6.2 \cite{C2}).
The following example, showing that limits might not exist on nonreduced local rings,  is Example  7.1 \cite{C2}.

\begin{Example}\label{Example1} Let $k$ be a field, $d>0$  and $R$ be the nonreduced $d$-dimensional local ring $R=k[[x_1,\ldots,x_d,y]]/(y^2)$.
There exists a  graded family of $m_R$-primary ideals $\{I_n\}$ in $R$ such that the limit
$$
\lim_{n\rightarrow \infty} \frac{\ell_R(R/I_n)}{n^d}
$$
does not exist, even when $n$ is constrained to lie in an arithmetic sequence.
\end{Example}

\begin{proof}  Let $\overline x_1, \ldots,\overline x_d,\overline y$ be the classes of $x_1,\ldots,x_d,y$ in $R$.
Let $N_i$ be the set of monomials of degree $i$ in the variables $\overline x_1,\ldots,\overline x_d$.
Let $\sigma(n)$ be the function defined in (\ref{eqsigma}).
Define $M_R$-primary ideals $I_n$ in $R$ by
 $I_n=(N_n,\overline yN_{n-\sigma(n)})$ for $n\ge 1$ (and $I_0=R$). 

We first verify that $\{I_n\}$ is a graded family of ideals, by showing that $I_mI_n\subset I_{m+n}$ for all $m,n>0$.
This follows since
$$
I_mI_n=(N_{m+n},\overline yN_{(m+n)-\sigma(m)}, \overline yN_{(m+n)-\sigma(n)})
$$
and  $\sigma(j)\le \sigma(k)$ for $k\ge j$.

$R/I_n$ has a $k$-basis consisting of 
$$
\{N_i\mid i<n\}\mbox{ and }\{\overline yN_j\mid j<n-\sigma(n)\}.
$$
Thus 
$$
\ell_R(R/I_n)=\binom{n}{d}+\binom{n-\sigma(n)}{d}.
$$
does not exist, even when $n$ is constrained to lie in an arithmetic sequence, by (\ref{eqnr1}).
\end{proof}

Hailong Dao and  and Ilya Smirnov have communicated to me that they have  also found this example, and have extended it to prove the following  theorem.

\begin{Theorem}\label{Theorem3}(Hailong Dao and Ilya Smirnov) Suppose that $R$ is a local ring of dimension $d>0$ with nilradical $N$. Suppose that for any graded family $\{I_i\}$ of $m_R$-primary ideals, the limit 
$$
\lim_{i\rightarrow\infty}\frac{\ell_R(R/I_n)}{n^d}
$$
exists.
Then $\dim N<d$.
\end{Theorem}

\begin{proof} Suppose that $\dim N=d$. Let $p$ be a minimal  prime of $N$ such that $\dim R/p=d$. Then $N_p\ne 0$, so $p_p\ne 0$ in $R_p$. $p$ is an associated prime of $N$,  so there exists $0\ne x\in R$ such that $\mbox{ann}(x)=p$. $x\in p$, since otherwise 
$0=pxR_p=p_p$ which is impossible. In particular, $x^2=0$.

Let $f(n)=n-\sigma(n)$ be the function of (\ref{eqsigma}), and define $m_R$-primary ideals in $R$ by
$$
I_n=m_R^n+xm_R^{f(n)}.
$$
$\{I_n\}$ is a graded family of ideals in $R$ since 
$$
I_mI_n=(m_R^{m+n},xm_R^{(m+n)-\sigma(m)}, xm_R^{(m+n)-\sigma(n)})
$$
and  $\sigma(j)\le \sigma(k)$ for $k\ge j$. Let $\overline R=R/xR$. We have short exact sequences
\begin{equation}\label{eqnr2}
0\rightarrow xR/xR\cap I_n\rightarrow R/I_n\rightarrow \overline R/I_n\overline R\rightarrow 0.
\end{equation}
By Artin-Rees, there exists a number $k$ such that
$xR\cap m_R^n=m_R^{n-k}(xR\cap m_R^{n-k})$ for $n>k$. Thus $xR\cap m_R^n\subset xm_R^{f(n)}$ for $n\gg 0$ and $xR\cap I_n=xm_R^{f(n)}$ for $n\gg 0$.
We have that
$$
xR/xR\cap I_n\cong xR/xm_R^{f(n)}\cong R/({\rm ann}(x)+m_R^{f(n)})\cong R/p+m_R^{f(n)},
$$
so that  $\ell_R(xR/xR\cap I_n)=P_{R/p}(f(n))$  for $n\gg0$, where $P_{R/p}(n)$ is the Hilbert-Samuel polynomial of $R/p$. Hence
\begin{equation}\label{eqnr3}
\lim_{n\rightarrow \infty}\frac{\ell_R(xR/xR\cap I_n)}{n^d}=\frac{e(m_{R/p})}{d!}\lim_{n\rightarrow\infty}\left(\frac{f(n)}{n}\right)^d
\end{equation}
does not exist by (\ref{eqnr1}).
For $n\gg0$, 
$$
\ell_R(\overline R/I_n\overline R)=\ell_R(\overline R/m_{\overline R}^n)=P_{\overline R}(n)
$$
where $P_{\overline R}(n)$ is the Hilbert-Samuel polynomial of $\overline R$. Since $\dim \overline R\le d$, we have that
\begin{equation}\label{eqnr4}
\lim_{n\rightarrow \infty}\frac{\ell_R(\overline R/I_n\overline R)}{n^d}
\end{equation}
exists. Thus
$$
\lim_{n\rightarrow \infty}\frac{\ell_R(R/I_n)}{n^d}
$$
does not exist by (\ref{eqnr2}), (\ref{eqnr3}) and (\ref{eqnr4}).
\end{proof}

\begin{Theorem}\label{Theorem4} Suppose that $R$ is a local ring of dimension $d$, and  $N$ is the nilradical of the $m_R$-adic completion $\hat R$ of $R$.  Then   the limit 
$$
\lim_{i\rightarrow\infty}\frac{\ell_R(R/I_n)}{n^d}
$$
exists for any graded family $\{I_i\}$ of $m_R$-primary ideals, if and only if $\dim N<d$.
\end{Theorem}

\begin{proof} Sufficiency  follows from Theorem \ref{Theorem2}. Necessity  follows from  Theorem \ref{Theorem3} if $d>0$, since a family of 
$m_{\hat R}$-primary ideals in $\hat R$ naturally lifts to a graded family of $m_R$-primary ideals in $R$.

In the case when $d=0$ and $N\ne 0$, $R$ is an Artin local ring. Thus there exists some number $0<t$ such that $m_R^t\ne 0$ but $m_R^{t+1}=0$. With the notation before (\ref{eqsigma}), let
\begin{equation}\label{eqtau}
\tau(n)=\left\{\begin{array}{ll} 0&\mbox{ if $i_j\le n\le i_{j+1}$ and $j$ is even}\\
1&\mbox{ if $i_j\le n\le i_{j+1}$ and $j$ is odd.}\end{array}\right.
\end{equation}
Define a graded family of $m_R$-primary ideals $\{I_n\}$ in $R$ by $I_n=m_R^{t+\tau(n)}$. Then $\lim_{n\rightarrow \infty}\ell_R(R/I_n)$ does not exist.

\end{proof}

\begin{Corollary}\label{Cornr20}
Suppose that $R$ is an excellent local ring of dimension $d$, and  $N(R)$ is the nilradical of  $R$.  Then   the limit 
$$
\lim_{i\rightarrow\infty}\frac{\ell_R(R/I_n)}{n^d}
$$
exists for any graded family $\{I_i\}$ of $m_R$-primary ideals, if and only if $\dim N(R)<d$.
\end{Corollary}

\begin{proof} Let $N(\hat R)$ be the nilradical of $\hat R$. $\widehat{(R/N(R)}\cong \hat R/N(R)\hat R$ is reduced since $R/N(R)$ is (by Scholie IV.7.8.3 \cite{EGAIV}). Since $N(R)\hat R\subset N(\hat R)$, we have that $N(\hat R)=N(R)\hat R$. Thus $\mbox{gr}_{m_{\hat R}}(N(\hat R)) = \mbox{gr}_{m_R}(N(R)$,
so $\dim N(\hat R)=\dim N(R)$. Now the corollary follows from Theorem \ref{Theorem4}.
\end{proof}

\begin{Example} For any $d\ge 1$, there exists a local domain $R$ of dimension $d$ with a graded family of $m_R$-primary ideals $\{I_n\}$ such that the limit
$$
\lim_{n\rightarrow \infty}\frac{\ell_R(R/I_n)}{n^d}
$$
does not exist.
\end{Example}
\begin{proof}
The example of (E3.2) in \cite{N} is of a local domain $R$ such that the nilradical of $\hat R$ has dimension $d$. The example then follows from
Theorem \ref{Theorem3}  by lifting an appropriate graded family of $m_{\hat R}$-primary ideals to $R$.
\end{proof}

In Section 4 of \cite{CDK}, a series of examples of graded families of $m_R$-primary ideals in a regular local ring $R$ of dimension two are given which have asymptotic growth of the rate $n^{\alpha}$, where $\alpha$ can be any rational number $1\le \alpha\le 2$. An example, also in a regular local ring of dimension two,  with growth  rate $n\mbox{log}_{10}(n)$ is given. Thus we generally do not have  a polynomial rate of growth.

\section{Proper $k$-schemes}\label{SecProp}

Suppose that $X$ is a $d$-dimensional proper scheme over a field $k$, and $\mathcal L$ is a line bundle on $X$. 
Then under the natural inclusion of rings $k\subset \Gamma(X,\mathcal O_X)$, we have that the section ring
$$
\bigoplus_{n\ge 0}\Gamma(X,\mathcal L^n)
$$
is a graded $k$-algebra. Each $\Gamma(X,\mathcal L^n)$ is a finite dimensional $k$-vector space since $X$ is proper over $k$. In particular, $\Gamma(X,\mathcal O_X)$ is an Artin ring.     A graded $k$-subalgebra $L=\bigoplus_{n\ge 0}L_n$ of a section ring of a line bundle $\mathcal L$ on $X$ is called a {\it graded linear series} for $\mathcal L$. 

 We define the {\it Kodaira-Iitaka dimension} $\kappa=\kappa(L)$ of a graded linear series $L$  as follows.
Let
$$
\sigma(L)=\max \left\{m\mid 
\begin{array}{l}
 \mbox{there exists $y_1,\ldots,y_m\in L$ which are homogeneous of  positive}\\
\mbox{degree and are algebraically independent over $k$}
\end{array}\right\}.
$$
$\kappa(L)$ is then defined as
$$
\kappa(L)=\left\{\begin{array}{ll}
\sigma(L)-1 &\mbox{ if }\sigma(L)>0\\
-\infty&\mbox{ if }\sigma(L)=0
\end{array}\right.
$$

This definition is in agreement with the classical definition for line bundles on projective varieties (Definition in Section 10.1 \cite{I} or Chapter 2 \cite{La}). Some of its properties and pathologies (on nonreduced schemes) are discussed in \cite{C2}. The following Lemma is proven in Lemma 2.1 \cite{C2} for projective schemes.

\begin{Lemma}\label{LemmaKI} Suppose that $L$ is a graded linear series on a $d$-dimensional proper scheme $X$ over a field $k$. Then
\begin{enumerate}\item[1)]
\begin{equation}\label{eqKI1}
\kappa(L)\le d=\dim X.
\end{equation}
\item[2)] There exists a positive constant $\gamma$ such that 
\begin{equation}\label{eqKI4}
\dim_k L_n<\gamma n^d
\end{equation}
for all $n$.
\item[3)] Suppose that $\kappa(L)\ge 0$. Then there exists a positive constant $\alpha$ and a positive integer $e$ such that 
\begin{equation}\label{eqKI2}
\dim_kL_{en}>\alpha n^{\kappa(L)}
\end{equation}
for all positive integers $n$. 
\item[4)] Suppose that $X$ is reduced and $L$ is a graded linear series on $X$. Then $\kappa(L)=-\infty$ if and only if $L_n=0$ for all $n>0$.
\end{enumerate}
\end{Lemma}

\begin{proof} The proof of statements 3) and 4) in Lemma 2.1 \cite{C2} do not use the assumption that $X$ is projective, so we need only establish 1) and 2).  
Formula 2) follows from the following formula: Suppose that $\mathcal M$ is a coherent sheaf on $X$. Then there exists a positive constant $\gamma$ such that 
\begin{equation}\label{eqF}
\dim_k \Gamma(X,\mathcal M\otimes \mathcal L^n)<\gamma n^d
\end{equation}
 for all positive $n$.
 
 We first prove (\ref{eqF}) when $X$ is projective. Let $\mathcal O_X(1)$ be a very ample line bundle on $X$. By Proposition 7.4 \cite{H}, there
 exists a finite filtration of $\mathcal M$ by coherent sheaves $\mathcal M^i$ with quotients $\mathcal M^i/\mathcal M^{i-1}\cong \mathcal O_{Y_i}(n_i)$,
 where $Y_i$ are closed integral subschemes of $X$ and $n_i\in \ZZ$. There exists a number $c>0$ such that $\mathcal L\otimes \mathcal O_{Y_i}(c)$ is
 ample for all $i$. Let $\mathcal A=\mathcal O_X(n)\otimes \mathcal L$,
 where $n=c+\max\{|n_i|\}$. For all $i$ and positive $n$, we have 
 $$
 \dim_k\Gamma(X,(\mathcal M_i/\mathcal M_{i-1})\otimes \mathcal L^n)\le \dim_k\Gamma(Y_i,\mathcal O_{Y_i}\otimes \mathcal A^n). 
 $$
 This last is a polynomial in $n$ of degree equal to $\dim Y_i$ for large $n$ (by Proposition 8.8a \cite{I}).
 Thus we obtain the formula (\ref{eqF}) in the case that $X$ is projective. 
 
 Now suppose that $X$ is proper over $k$. We prove the formula by induction on $\dim \mathcal M$. If $\dim \mathcal M=0$, then $\dim_k\Gamma(X,\mathcal M)<\infty$, and $\mathcal M\otimes\mathcal L^n\cong\mathcal M$ for all $n$, so (\ref{eqF}) holds.
 Suppose that $\dim \mathcal M=e\,(\le d)$ and the formula is true for coherent $\mathcal O_X$-modules whose support has dimension $<e$. Let $\mathcal I$ be the sheaf of ideals on $X$ defined for $\eta\in X$ by
 $$
 \mathcal I_{\eta}=\{f\in \mathcal O_{X,\eta}\mid f\mathcal M_{\eta}=0\}.
 $$
 Let $Y=\mbox{Spec}(\mathcal O_X/\mathcal I)$, a closed subscheme of $X$. $\mathcal M$ is a coherent $\mathcal O_Y$-module, and
 $Y$ and $\mathcal M$ have the same support, so $\dim Y=e$. By Chow's Lemma, there exists a proper morphism
 $\phi:Y'\rightarrow Y$ such that $Y'$ is projective over $k$ and $\phi$ is an isomorphism over an open dense subset of $Y$. Ket $\mathcal K$ be the kernel of the natural morphism of $\mathcal O_Y$-modules
 $$
 \mathcal M\rightarrow \phi_*\phi^*\mathcal M.
 $$
 $\dim\mathcal K<e$ since $\phi$ is an isomorphism over a dense open subset of $Y$. Let $\mathcal L'=\phi^*(\mathcal L\otimes\mathcal O_Y)$. We have inequalities
 $$
 \dim_k\Gamma(X,\mathcal M\otimes\mathcal L^n)\le \dim_k\Gamma(X,\mathcal K\otimes\mathcal L^n)+\dim_k\Gamma(Y',\phi^*(\mathcal M)\otimes(\mathcal L')^n)
 $$
 for all $n\ge 0$, so we get the desired upper bound of (\ref{eqF}).

Now we establish 1). Let $\kappa:=\kappa(L)$. Then there exists an inclusion of a weighted polynomial ring $k[x_0,\ldots,x_{\kappa}]$ into $L$. 
Let $f$ be the least common multiple of the degrees of the $x_i$. Let $Z=\mbox{Proj}(k[x_0,\ldots,x_{\kappa}])$. $\mathcal O_Z(f)$ is an ample line bundle on the $\kappa$-dimensional weighted projective space $Z$. Thus there exists a polynomial $Q(n)$ of degree $\kappa$ such that
$$
\dim_kk[x_0,\ldots,x_n]_{nf}=\dim_k\Gamma(Z,\mathcal O_Z(nf))=Q(n)\mbox{ for all }n\gg 0.
$$
Thus there exists a positive constant $\alpha$ such that 
$$
\dim_kL_{nf}\ge \alpha n^{\kappa}\mbox{ for }n\gg 0,
$$
whence $\kappa(L)\le d$ by 2).

\end{proof}

We will show in Theorem \ref{Theorem8} that (\ref{eqKI4}) of Lemma \ref{LemmaKI} can be sharpened to the statement that there exists a positive constant $\gamma$ such that
\begin{equation}\label{eqN1}
\dim_kL_n<\gamma n^e
\end{equation}
where $e=\max\{\kappa(L),\dim \mathcal N_X\}$. $\mathcal N_X$ is the nilradical of $X$ (defined in the section on notations and conventions).
By Theorem \ref{TheoremN1}, (\ref{eqN1}) is the best bound possible.

\section{Limits of graded linear series on proper varieties over a field}\label{SecLim}

Suppose that $L$ is a graded linear series on a proper variety $X$ over a field $k$. 
The {\it index} $m=m(L)$ of $L$ is defined as the index of groups
$$
m=[\ZZ:G]
$$
where $G$ is the subgroup of $\ZZ$ generated by $\{n\mid L_n\ne 0\}$.

The following theorem has been proven by  Okounkov  \cite{Ok} for section rings of ample line bundles,  Lazarsfeld and Mustata \cite{LM} for section rings of big line bundles, and for graded linear series by Kaveh and Khovanskii \cite{KK}. All of these proofs require the assumption that {\it $k$ is  algebraically closed}. The theorem has been proven by the author when $k$ is a perfect field in \cite{C2}. We prove the result here for an arbitrary base field $k$.

\begin{Theorem}\label{Theorem5}  Suppose that $X$ is a $d$-dimensional proper variety over a field $k$, and $L$ is a graded linear series on $X$ with Kodaira-Iitaka dimension  $\kappa=\kappa(L)\ge 0$. Let $m=m(L)$ be the index of $L$.  Then  
$$
\lim_{n\rightarrow \infty}\frac{\dim_k L_{nm}}{n^{\kappa}}
$$
exists. 
\end{Theorem}

In particular, from the definition of the index, we have that the limit
$$
\lim_{n\rightarrow \infty}\frac{\dim_k L_{n}}{{n}^{\kappa}}
$$
exists, whenever $n$ is constrained to lie in an arithmetic sequence $a+bm$ ($m=m(L)$ and $a$ an arbitrary but fixed constant), as $\dim_kL_n=0$ if $m\not\,\mid n$.

An example of a big line bundle where the limit in Theorem \ref{Theorem5} is an irrational number is given in Example 4 of Section 7 \cite{CS}.

The following theorem is proven by Kaveh and Khovanskii \cite{KK} when $k$ is an algebraically closed field (Theorem 3.3 \cite{KK}). We prove the theorem for an arbitrary field.

\begin{Theorem}\label{Theorem100} Suppose that $X$ is a $d$-dimensional proper variety over a field $k$, and $L$ is a graded linear series on $X$ with Kodaira-Iitaka dimension $\kappa=\kappa(L)\ge 0$. Let $m=m(L)$ be the index of $L$.  
Let $Y_{nm}$ be the projective subvariety of $\PP^{\dim_kL_{nm}}$ that is the closure of the image of the rational map 
$L_{nm}:X \dashrightarrow \PP_k^{\dim_kL_{nm}}$. Let $\deg(Y_{nm})$ be the degree of $Y_{nm}$ in $\PP_k^{\dim_kL_{nm}}$.
Then  $\dim Y_{nm}=\kappa$ for $n\gg 0$ and 
$$
\lim_{n\rightarrow \infty}\frac{\dim_k L_{nm}}{n^{\kappa}}=\lim_{n\rightarrow\infty}\frac{\deg(Y_{nm})}{\kappa!n^{\kappa}}.
$$
\end{Theorem}

Letting $t$ be an indeterminate, $\deg(Y_{nm})$ is the multiplicity of the  graded $k$-algebra $k[L_{nm}t]$ (with elements of $L_{nm}t$ having degree 1). 

We now proceed to prove Theorems \ref{Theorem5} and \ref{Theorem100}.

By Chow's Lemma, there exists a proper birational morphism $\phi:X'\rightarrow X$ which is an isomorphism over a dense open set, such that $X'$ is projective over $k$. Since $X$ is integral, we have an inclusion $\Gamma(X,\mathcal L^n)\subset \Gamma(X,\phi^*\mathcal L^n)$ for all $n$.
Thus $L$ is a graded linear series for $\phi^*\mathcal L$, on the projective variety $X'$. In this way, we can assume that $X$ is in fact projective over $k$.

By \cite{Z}, $X$ has a closed regular point $Q$ (even though there may be no points which are smooth over $k$ if $k$ is not perfect).
Let  $R=\mathcal O_{X,Q}$. $R$ is a $d$-dimensional regular local ring. Let $k'=k(Q)=R/m_R$.

Choose a regular system of parameters $y_1,\ldots, y_d$ in $R$. By a similar  argument to that of the proof of Theorem \ref{Theorem1}, we may
define a valuation $\nu$ of the function field $k(X)$ of $X$ dominating $R$, by stipulating that
 \begin{equation}\label{eq20}
\nu(y_i)= e_i\mbox{ for $1\le i\le d$}
\end{equation}
where $\{e_i\}$ is the standard basis of the totally ordered  group $\Gamma_{\nu}=(\ZZ^d)_{\rm lex}$, and
 $\nu(c)=0$ if $c$ is a unit in $R$. As  in the proof of Theorem \ref{Theorem1}, we have that the residue field of the valuation ring $V_{\nu}$ of $\nu$ is  $V_{\nu}/m_{\nu}=k(Q)=k'$.

$L$ is a graded linear series for some line bundle $\mathcal L$ on $X$.  Since $X$ is integral, $\mathcal L$ is isomorphic to an invertible sheaf
$\mathcal O_X(D)$ for some Cartier divisor $D$ on $X$. 
We can assume that $Q$ is not contained in the support of $D$, after possibly replacing $D$ with a Cartier divisor linearly equivalent to $D$.
We have an induced graded $k$-algebra isomorphism of section rings 
$$
\bigoplus_{n\ge 0}\Gamma(X,\mathcal L^n)\rightarrow \bigoplus_{n\ge 0}\Gamma(X,\mathcal O_X(nD))
$$
which takes $L$ to a graded linear series for $\mathcal O_X(D)$. Thus we may assume that $\mathcal L=\mathcal O_X(D)$.
 For all $n$, the restriction map followed by inclusion into $V_{\nu}$,
\begin{equation}\label{eqR3}
\Gamma(X,\mathcal L^n)\rightarrow \mathcal L_Q=\mathcal O_{X,Q}\subset V_{\nu}
\end{equation}
 is a 1-1 $k$-vector space homomorphism since $X$ is integral, and we have an induced $k$-algebra homomorphism 
$$
 L\rightarrow \mathcal O_{X,Q}\subset V_{\nu}.
 $$

Given a nonnegative element $\gamma$ in the  value group $\Gamma_{\nu}=(\ZZ^d)_{\rm lex}$ of $\nu$, we have associated valuation ideals $I_{\gamma}$ and $I_{\gamma}^+$ in $V_{\nu}$ defined by 
$$
I_{\gamma}=\{f\in V_{\nu}\mid \nu(f)\ge \gamma\}
$$
and
$$
I_{\gamma}^+=\{f\in V_{\nu}\mid \nu(f)>\gamma\}.
$$
Since $V_{\nu}/m_{\nu}=k'$,  we have  that $I_{\lambda}/I_{\lambda}^+\cong k'$ for all nonnegative elements $\lambda\in \Gamma_{\nu}$, so
\begin{equation}\label{eqR1}
\dim_k(I_{\gamma}/I_{\gamma}^+)= [k':k]< \infty
\end{equation}
for all non negative $\gamma \in \Gamma_{\nu}$.
For $1\le t$, let
$$
S(L)_n^{(t)}=\{\gamma\in \Gamma_{\nu}\mid \dim_k L_n\cap I_{\gamma}/L_n\cap I_{\gamma}^+\ge t\}.
$$
Since every element of $L_n$ has non negative value (as $L_n\subset V_{\nu}$), we have by (\ref{eqR1}) and  (\ref{eqR3}) that 
\begin{equation}\label{eqR2}
\dim_kL_n=\sum_{t=1}^{[k':k]}\#(S(L)_n^{(t)})
\end{equation}
for all $n$.
For $1\le t$, let 
$$
S(L)^{(t)}=\{(\gamma,n)|\gamma \in S(L)_n^{(t)}\}.
$$
We have inclusions of semigroups
$S(L)^{(t')}\subset S(L)^{(t)}$ if $t<t'$.

\begin{Lemma}\label{Lemmaproj1} Suppose that $t\ge 1$, $0\ne f\in L_i$, $0\ne g\in L_j$ and
$$
\dim_kL_i\cap I_{\nu(f)}/L_i\cap I_{\nu(f)}^+\ge t.
$$
Then
\begin{equation}\label{eqproj10}
\dim_k L_{i+j}\cap I_{\nu(fg)}/L_{i+j}\cap I_{\nu(fg)}^+\ge t.
\end{equation}
In particular,
 the $S(L)^{(t)}$  are subsemigroups of the  semigroup $\ZZ^{d+1}$ whenever $S(L)^{(t)}\ne \emptyset$. We have that $m(S(L)^{(t)})=m(S(L)^{(1)})$ and $q(S(L)^{(t)})=q(S(L)^{(1)})$ for all $t$ such that $S(L)^{(t)}\not\subset  \{0\}$.
\end{Lemma}

\begin{proof}
The proof of (\ref{eqproj10}) and that $S(L)^{(t)}$ are sub semigroups is similar to that of Lemma \ref{Lemmared1}.

Suppose that $S(L^{(t)})\not\subset  \{0\}$. $S(L)^{(t)}\subset S(L)^{(1)}$, so 
\begin{equation}\label{eqnr60}
m(S(L)^{(1)})\mbox{ divides }m(S(L)^{(t)})
\end{equation}
and
\begin{equation}\label{eqnr61}
q(S(L)^{(t)})\le q(S(L)^{(1)}).
\end{equation}

For all $a\gg 0$, $S(L)^{(1)}_{am(S(L)^{(1)})}\ne \emptyset$. In particular, we can take  $a\equiv 1\,((\mbox{mod }mS(L)^{(t)}))$. There exists $b>0$ such that
$S(L)^{(t)}_{bm(S(L)^{(t)})}\ne \emptyset$. By (\ref{eqproj10}), we have that 
$$
S(L)^{(t)}_{am(S(L)^{(1)})+bm(S(L)^{(t)})}\ne \emptyset.
$$
Thus $m(S(L)^{(1)})\in \pi(G(S(L)^{(t)}))$, and by (\ref{eqnr60}), $m(S(L)^{(t)})=m(S(L)^{(1)})$.

Let $q=q(S(L)^{(1)})$. There exists $n_1>0$ and $(\gamma_1,n_1),\ldots, (\gamma_q,n_1)\in S(L)^{(1)}_{n_1}$ 
 such that if $\mathcal C_1$ is the cone generated by $(\gamma_1,n_1),\ldots, (\gamma_q,n_1)$ in $\RR^{d+1}$, then 
 $\dim \mathcal C_1\cap(\RR^d\times \{1\})=q$. There exists $(\tau, n_2)\in S(L^{(t)})$ with $n_2>0$. Thus 
 $$
 (\tau+\gamma_1, n_1+n_2),\ldots, (\tau+\gamma_q,n_1+n_2)\in S(L)^{(t)}_{n_1+n_2}
 $$
 by (\ref{eqproj10}). Let $\mathcal C_2$ be the cone generated by 
 $$
 (\tau+\gamma_1, n_1+n_2),\ldots, (\tau+\gamma_q,n_1+n_2)
 $$
 in $\RR^{d+1}$. Then $\dim \mathcal C_2\cap(\RR^d\times\{1\})=q$, and $q\le q(S(L)^{(t)})$. Thus, by (\ref{eqnr61}), $q(S(L)^{(t)})=q(S(L)^{(1)})$.

\end{proof}

We have that $m=m(L)$ is the common value of  $m(S(L)^{(t)})$. Let $q(L)$ be the common value of $q(S(L)^{(t)})$ for  $S(L)^{(t)}\not\subset \{0\}$.

There exists a very ample Cartier divisor $H$ on $X$ (at the beginning of the proof we reduced to $X$ being projective) such that $\mathcal O_X(D)\subset \mathcal O_X(H)$ and the point $Q$ of $X$ (from the beginning of the proof)
is  not contained in the support of $H$. Let $A_n=\Gamma(X,\mathcal O_X(nH))$ and $A$ be the section ring
$A=\bigoplus_{n\ge 0}A_n$. After possibly replacing $H$ with a sufficiently high multiple of $H$, we may assume that $A$ is generated in degree 1 as a $k''=\Gamma(X,\mathcal O_X)$-algebra. $[k'':k]<\infty$ since $X$ is projective. The $k$-algebra homomorphism $L\rightarrow V_{\nu}$ defined after (\ref{eqR3}) extends to a $k$-algebra homomorphism
$L\subset A\rightarrow V_{\nu}$. Let
$$
T_n=\{\gamma\in \Gamma_{\nu}\mid A_n\cap I_{\gamma}/A_n\cap I_{\gamma}^+\ne 0\},
$$
and $T=\{(\gamma,n)\mid \gamma\in T_n\}$. 
$T$ is a subsemigroup of $\ZZ^{d+1}$ by the argument of Lemma \ref{Lemmaproj1}, and we have inclusions of semigroups $S^{(t)}\subset T$ for all $t$.

By our construction, $A$ is naturally a graded subalgebra of the graded algebra $\mathcal O_{X,Q}[t]$. Since $H$ is ample on $X$, we have that $A_{(0)}=k(X)$, where $A_{(0)}$ is the set of elements of degree 0 in the localization of $A$ at the set of nonzero homogeneous elements of $A$. Thus for $1\le i\le d$, there exists $f_i,g_i\in A_{n_i}$, for some $n_i$, such that $\frac{f_i}{g_i}=y_i$. Thus
$$
(e_i,0)=(\nu(y_i),0)=(\nu(f_i),n_i)-(\nu(g_i),n_i)\in G(T).
$$
for $1\le i\le d$. Since $A_1\ne 0$, we then have that $(0,1)\in G(T)$. Thus $G(T)=\ZZ^{d+1}$, so $L(T)=\RR^{d+1}$, $\partial M(T)=\RR^d\times\{0\}$ and $q(T)=d$ (with the notation of  Section \ref{SecCone} on cones and semigroups).
For all $n\gg 0$ we have a bound
$$
|T_n|\le \dim_kA_n=[k'':k]\dim_{k''}A_n= [k'':k]P_A(n)
$$
where $P_A(n)$ is the Hilbert polynomial of the $k''$-algebra $A$, which has degree $\dim X=d=q(T)$. Thus, by Theorem 1.18 \cite{KK}, $T$ is a strongly nonnegative semigroup. Since the $S(L)^{(t)}$ are subsemigroups of $T$, they are also strongly nonnegative,
 so by Theorem \ref{ConeTheorem3} and (\ref{eqR2}), we have that
\begin{equation}\label{eqnr80}
\lim_{n\rightarrow \infty}\frac{\dim_kL_{nm}}{n^{q(L)}}=\sum_{t=1}^{[k':k]}\lim_{n\rightarrow \infty}\frac{\#(S(L)_{nm}^{(t)})}{n^{q(L)}}
=\sum_{t=1}^{[k':k]}\frac{{\rm vol}_{q(L)}(\Delta(S(L)^{(t)}))}{{\rm ind}(S(L)^{(t)})}
\end{equation}
exists.

Let $Y_{pm}$ be the varieties defined in the statement of Theorem \ref{Theorem100}. Let $d(pm)=\dim Y_{pm}$. The coordinate ring of $Y_{pm}$ is the $k$-subalgebra $L^{[pm]}:= k[L_{pm}]$ of $L$ (but with the grading giving elements of $L_{pm}$ degree 1). The Hilbert polynomial $P_{Y_{pm}}(n)$ of $Y_{pm}$ (Section I.7 \cite{H} or Theorem 4.1.3 \cite{BH}) has the properties that
\begin{equation}\label{eqred90}
P_{Y_{pm}}(n)=\frac{\deg(Y_{pm})}{d(pm)!}n^{d(pm)}+\mbox{lower order terms}
\end{equation}
and
\begin{equation}\label{eqred61}
\dim_kL^{[pm]}_{npm}=P_{Y_{pm}}(n)
\end{equation}
for $n\gg 0$. We have that
\begin{equation}\label{eqnr81}
\lim_{n\rightarrow \infty}\frac{\dim_k(L^{[pm]})_{npm}}{n^{d(pm)}}=\frac{\deg(Y_{pm})}{d(pm)!}.
\end{equation}

Suppose that $t$ is such that $1\le t\le [k':k]$ and $S(L)^{(t)}\not\subset \{ 0\}$. By Lemma \ref{Lemmaproj1}, for $p$ sufficiently large, we have that $m(S(L^{[pm]})^{(t)})=mp$. Let $\mathcal C$ be the closed cone generated by $S(L)^{(t)}_{pm}$ in $\RR^{d+1}$. We also have that 
$$
\dim (\mathcal C\cap (\RR^d\times\{1\}))=\dim(\Delta(S(L)^{(t)})=q(L)
$$
 for $p$ sufficiently large (the last equality is by Lemma \ref{Lemmaproj1}). Since
$S(L)^{(t)}_{pm}=S(L^{[pm]})^{(t)}_{pm}$, we have that 
$$
\dim (\mathcal C\cap (\RR^d\times\{1\}))\le \dim(\Delta(S(L^{[pm]})^{(t)})\le \dim (\Delta(L)^{(t)}).
$$
Thus 
\begin{equation}\label{eqred30}
q(S(L^{[pm]})^{(t)})=q(L)
\end{equation}
for all $p$ sufficiently large.

By the definition of Kodaira-Iitaka dimension, we also have that
\begin{equation}\label{eqred31} 
\kappa(L^{[pm]})=\kappa(L)
\end{equation}
for $p$ sufficiently large.

Now by graded Noether normalization (Section I.7 \cite{H} or Theorem 1.5.17\cite{BH}), the finitely generated algebra $k$-algebra $L^{[pm]}$ satisfies
\begin{equation}\label{eqred63}
d(pm)=\dim Y_{pm}=\mbox{Krull dimension}(L^{[pm]})-1=\kappa(L^{[pm]}).
\end{equation}

 We have that
\begin{equation}\label{eqred62}
\frac{1}{[k':k]}\dim_k L^{[pm]}_{npm}\le \#(S(L^{[pm]})^{(1)}_{npm})\le \dim_kL^{[pm]}_{npm}
\end{equation}
for all  $n$. $S(L^{[pm]})^{(1)}$ is strongly nonnegative since $S(L^{[pm]})^{(1)}\subset S(L)^{(1)}$ (or since $L^{[pm]}$ is a finitely generated $k$-algebra).
It follows from Theorem \ref{ConeTheorem3}, (\ref{eqred62}), (\ref{eqred61}), (\ref{eqred90}) and (\ref{eqred63})  that 
\begin{equation}\label{eqred91}
q(S(L^{[pm]})^{(1)})=d(pm)=\kappa(L^{[pm]}).
\end{equation}
From (\ref{eqred30}), (\ref{eqred91}) and (\ref{eqred31}), we have that 
\begin{equation}\label{eqnr85}
q(L)=\kappa(L)=\kappa.
\end{equation}

Theorem \ref{Theorem5} now follows from (\ref{eqnr80}) and (\ref{eqnr85}).
We now prove Theorem \ref{Theorem100}.  For all $p$, we have inequalities
$$
\sum_{t=1}^{[k':k]}\#(n*S(L)^{(t)}_{mp})\le \sum_{t=1}^{[k':k]}\#(S(L^{[mp]})^{(t)}_{nmp})\le \sum_{t=1}^{[k':k]}\#(S(L)_{nmp}^{(t)}).
$$
The second term in the inequality is $\dim_k(L^{[pm]})_{nmp}$ and the third term is $\dim_kL_{nmp}$.
Dividing by $n^{\kappa}p^{\kappa}$, and taking the limit as $n\rightarrow \infty$, we obtain from 
Theorem \ref{ConeTheorem4}, (\ref{eqnr85}) and (\ref{eqnr80}) for the first term and (\ref{eqnr81}), (\ref{eqred31}) and (\ref{eqred63}) for the second term,
 that
for given $\epsilon >0$, we can take $p$ sufficiently large that
$$
\lim_{n\rightarrow\infty}\frac{\dim_kL_{nm}}{n^{\kappa}}-\epsilon\le \frac{\deg(Y_{pm})}{\kappa!p^{\kappa}}\le 
\lim_{n\rightarrow\infty}\frac{\dim_kL_{nm}}{n^{\kappa}}.
$$
Taking the limit as $p$ goes to infinity then proves Theorem \ref{Theorem100}.

\section{Limits on reduced proper schemes over a field}\label{SecRed}

Suppose that $X$ is a proper scheme over a field $k$ and $L$ is a graded linear series for a line bundle $\mathcal L$ on $X$.
Suppose that $Y$ is a closed subscheme of $X$. Set $\mathcal L|Y=\mathcal L\otimes_{\mathcal O_X}\mathcal O_Y$. Taking global sections of the natural surjections
$$
\mathcal L^n\stackrel{\phi_n}{\rightarrow} (\mathcal L|Y)^n\rightarrow 0,
$$
for $n\ge 1$ we have  induced short exact sequences of $k$-vector spaces
\begin{equation}\label{eq54}
0\rightarrow K(L,Y)_n\rightarrow L_n\rightarrow (L|Y)_n\rightarrow 0,
\end{equation}
where 
$$
(L|Y)_n:=\phi_n(L_n)\subset \Gamma(Y,({\mathcal L}|Y)^n)
$$
 and $K(L ,Y)_n$ is the kernel of $\phi_n|L_n$. Defining $K(L,Y)_0=k$ and $(L|Y)_0=\phi_0(L_0)$, we have that
$L|Y=\bigoplus_{n\ge 0}(L|Y)_n$ is a graded linear series for $\mathcal L|Y$ and $K(L,Y)=\bigoplus_{n\ge 0}K(L,Y)_n$ is a graded linear series for $\mathcal L$.

\begin{Lemma}\label{Lemma50a}(Lemma 5.1 \cite{C2})  Suppose that $X$ is a reduced proper scheme over a field $k$ and $X_1,\ldots,X_s$ are the irreducible components of $X$. Suppose that $L$ is a graded linear series on $X$. Then 
$$
\kappa(L)=\max\{\kappa(L|X_i)\mid 1\le i\le s\}.
$$
\end{Lemma}

The following theorem is proven in Theorem 5.2 \cite{C2} for reduced projective schemes over a perfect field.

\begin{Theorem}\label{Theorem18} Suppose that $X$ is a reduced proper scheme over a  field $k$. 
Let $L$ be a graded linear series on $X$ with Kodaira-Iitaka dimension  $\kappa=\kappa(L)\ge 0$.  
 Then there exists a positive integer $r$ such that 
$$
\lim_{n\rightarrow \infty}\frac{\dim_k L_{a+nr}}{n^{\kappa}}
$$
exists for any fixed $a\in \NN$.
\end{Theorem}
The theorem says that 
$$
\lim_{n\rightarrow \infty}\frac{\dim_k L_{n}}{n^{\kappa}}
$$
exists if $n$ is constrained to lie in an arithmetic sequence $a+br$ with $r$ as above, and for some fixed $a$. The conclusions of the theorem are a little weaker than the conclusions of Theorem \ref{Theorem5} for  varieties. In particular, the index $m(L)$ has little relevance on reduced but not irreducible schemes (as shown by the example after Theorem \ref{Theorem8} and Example 5.5 \cite{C2}).

\begin{proof}  Let $X_1,\ldots,X_s$ be the irreducible components of $X$. 
 Define graded linear series $M^i$ on $X$
 by $M^0=L$, $M^i=K(M^{i-1},X_i)$ for $1\le i\le s$. 
 By (\ref{eq54}), for $n\ge 1$, we have exact sequences of $k$-vector spaces 
 $$
 0\rightarrow (M^{j+1})_n=K(M^j,X_{j+1})_n\rightarrow M_n^j\rightarrow (M^j|X_{j+1})_n\rightarrow 0
 $$
 for $0\le j\le s-1$, and thus
 $$
 M_n^j={\rm Kernel}(L_n\rightarrow \bigoplus_{i=1}^j(L|X_i)_n)
 $$
 for $1\le j\le s$. The natural map $L\rightarrow \bigoplus_{i=1}^sL|X_i$ is an injection of $k$-algebras since 
 $X$ is reduced. Thus $M_n^s=(0)$, and
 \begin{equation}\label{eq71}
 \dim_kL_n=\sum_{i=1}^{s}\dim_k (M^{i-1}|X_i)_n
 \end{equation}
 for all $n$.
  Let $r=\mbox{LCM}\{m(M^{i-1}|X_i)\mid  \kappa(M^{i-1}|X_i)=\kappa(L)\}$. 
 The theorem now follows from Theorem \ref{Theorem5} applied to each of the $X_i$ with $\kappa(M^{i-1}|X_i)=\kappa(L)$ (we can start with an $X_1$ with $\kappa(L|X_1)=\kappa(L)$).
\end{proof}

\section{Necessary and sufficient conditions for limits to exist on a proper scheme over a field}\label{SecNec}

The nilradical $\mathcal N_X$ of a scheme $X$ is defined in the section on notations and conventions.

\begin{Lemma}\label{Lemmanr90} Suppose that $X$ is a proper scheme over a field $k$ and $L$ is a graded linear series on $X$. Then $\kappa(L)=\kappa(L|X_{red})$.
\end{Lemma}

\begin{proof} Let $\mathcal L$ be a line bundle associated to $X$. We have a commutative diagram
$$
\begin{array}{lllllllll}
&&0&&0&&0&&\\
&&\downarrow&&\downarrow&&\downarrow&&\\
0&\rightarrow & \bigoplus_{n\ge 0}K_n&\rightarrow & \bigoplus_{n\ge 0}L_n&\rightarrow & \bigoplus_{n\ge 0}(L|X_{\rm red})_n&\rightarrow &0\\
&&\downarrow&&\downarrow&&\downarrow&&\\
0&\rightarrow & \bigoplus_{n\ge 0} \Gamma(X,\mathcal L^n\otimes \mathcal N_X)&\rightarrow &\bigoplus_{n\ge 0}\Gamma(X,\mathcal L^n)&\rightarrow &
 \bigoplus_{n\ge 0}\Gamma(X_{\rm red},(\mathcal L|X_{\rm red})^n)
 \end{array}
 $$
 so that $K_n=L_n\cap \Gamma(X,\mathcal L^n\otimes\mathcal N_X)$ for all $n$.
 
 Suppose that $\sigma\in \Gamma(X,\mathcal L^m)$. $X$ is Noetherian, so there exists $r_0=r_0(\sigma)$ such that the closed sets $\mbox{sup}(\sigma^r)=\mbox{sup}(\sigma^{r_0})$ for all $r\ge r_0$. Thus 
 $$
 \begin{array}{l}
 \mbox{$\sigma\in \Gamma(X,\mathcal L^n\otimes\mathcal N_X)$ if and only if}\\
 \mbox{$\sigma_Q$ is torsion in the $\mathcal O_{X,Q}$-algebra $\bigoplus_{n\ge 0}\mathcal L_Q^n$ for all $Q\in X$, if and only if}\\
 \mbox{$\sigma_Q^{r_0}=0$ in $\bigoplus_{n\ge 0}\mathcal L^n_Q$ for all $Q\in X$, if and only if}\\
 \mbox{$\sigma^{r_0}=0$ in $\bigoplus_{n\ge 0}\Gamma(X,\mathcal L^n)$ since $\mathcal L$ is a sheaf.}
 \end{array}
 $$
 Thus $\bigoplus_{n\ge 0}\Gamma(X,\mathcal L^n\otimes \mathcal N_X)$ is the nilradical of $\bigoplus_{n\ge 0}\Gamma(X,\mathcal L^n)$ and 
 so $K$ is the nilradical of $L$.
 
 We have that $\kappa(L|X_{\rm red})\le \kappa(L)$ since any injection of a weighted polynomial ring into $L|X_{\rm red}$ lifts to a graded injection into $L$.
 
 If $A$ is a weighted polynomial ring which injects into $L$, then it intersects $K$ in $(0)$, so there is an induced graded inclusion of $A$ into $L|X_{\rm red}$. Thus $\kappa(L|X_{\rm red})=\kappa(L)$.

\end{proof}

\begin{Theorem}\label{Theorem8} 
Suppose that $X$ is a  proper scheme over a field $k$. Let $\mathcal N_X$ be the nilradical of $X$. Suppose that $L$ is a graded linear series on $X$. Then
\begin{enumerate}
\item[1)]  There exists a positive constant $\gamma$ such that $\dim_kL_n<\gamma n^e$ where 
$$
e=\max\{\kappa(L),\dim \mathcal N_X\}.
$$
\item[2)] Suppose that $\dim \mathcal N_X<\kappa(L)$. Then there exists a positive integer $r$ such that 
$$
\lim_{n\rightarrow \infty}\frac{\dim_kL_{a+nr}}{n^{\kappa(L)}}
$$
exists for any fixed $a\in\NN$. 
\end{enumerate}
\end{Theorem}

\begin{proof} Let $\mathcal L$ be a line bundle associated to $L$, so that $L_n\subset \Gamma(X,\mathcal L^n)$ for all $n$.
Let $K_n$ be the kernel of the surjection $L_n\rightarrow (L|X_{\rm red})_n$.
From the exact sequence
$$
0\rightarrow \mathcal N_X\rightarrow \mathcal O_X\rightarrow \mathcal O_{X_{\rm red}}\rightarrow 0,
$$ 
we see that $K_n\subset \Gamma(X,\mathcal N_X\otimes\mathcal L^n)$ for all $n$. There exists a constant $c$ such that
$$
\dim_k\Gamma(X,\mathcal N_X\otimes\mathcal L^n)<cn^{\dim \mathcal N_X}
$$
for all $n$. By Lemma \ref{Lemmanr90} and Theorem \ref{Theorem18}, 1. holds and there exists a positive integer $r$ such that for any $a$,
$$
\lim_{n\rightarrow \infty}\frac{\dim_k(L|X_{\rm red})_{a+nr}}{n^{\kappa(L)}}
$$
exists. Thus the conclusions of the theorem hold.
\end{proof}

An  example showing that the $r$ of the theorem might have to be strictly larger than the index $m(L)$ is obtained as follows. Let $X_1$ and $X_2$ be two general linear subspaces of dimension $d$  in $\PP^{2d}$. 
They intersect transversally in a rational point $Q$.  Let 
Let $L^i$ be the graded linear series on $X_i$ defined by 
$$
L^1_n=\left\{\begin{array}{ll}
\Gamma(X_1,\mathcal O_{X_1}(n)\otimes \mathcal O_{X_1}(-Q))&\mbox{ if } 2\mid n\\
0&\mbox{ otherwise}\end{array}\right.
$$
and
$$
L^2_n=\left\{\begin{array}{ll}
\Gamma(X_2,\mathcal O_{X_2}(n)\otimes \mathcal O_{X_2}(-Q))&\mbox{ if } 3\mid n\\
0&\mbox{ otherwise}\end{array}\right.
$$
Here $\mathcal O_{X_i}(-Q)$ denotes the ideal sheaf on $X_i$ of the point $Q$. Let $X$ be the reduced scheme whose support is $X_1\cup X_2$.
From the short exact sequence
$$
0\rightarrow \mathcal O_X\rightarrow \mathcal O_{X_1}\bigoplus \mathcal O_{X_1}\rightarrow k(Q)\rightarrow 0,
$$
 we see that there is a graded linear series $L$ on $X$ associated to $\mathcal O_X(1)$ such that
 $L|X_i=L^i$ for $i=1,2$, and $\dim_kL_n= \dim_kL^1_n+\dim_kL^2_n$ for all $n$. Thus
 $$
 \dim_k L_n=\left\{\begin{array}{cl}
 2\binom{d+n}{d}-2&\mbox{ if }n\equiv 0\, (\mbox{mod } 6)\\
 0&\mbox{ if }n\equiv 1\mbox{ or }5\,(\mbox{mod }6)\\
 \binom{d+n}{d}-1&\mbox{ if }n\equiv 2,3\mbox{ or }4\,(\mbox{mod }6).\\
 \end{array}\right.
 $$

In Section 6 of \cite{C2}, we give examples of graded linear series on nonreduced projective schemes, such that the limit
$$
\lim_{n\rightarrow \infty}\frac{\dim_k L_n}{n^r}
$$
does not exist (for suitable $r$), even when $n$ is constrained to lie in any arithmetic sequence. Some of the examples are section rings of line bundles.

In Theorem \ref{TheoremN1}, we give general conditions under which limits do not always exist.

\begin{Theorem}\label{TheoremN1} Suppose that $X$ is a $d$-dimensional projective scheme over a field $k$ with $d>0$.  Let $r=\dim \mathcal N_X$, where $\mathcal N_X$ is the nilradical of $X$. Suppose that $r\ge 0$. Let $s\in \{-\infty\}\cup \NN$ be such that $s\le r$. Then there exists a graded linear series $L$ on $X$ with $\kappa(L)=s$ such that 
$$
\lim_{n\rightarrow\infty} \frac{\dim_k L_n}{n^r}
$$
does not exist, even when $n$ is constrained to lie in any arithmetic sequence.
\end{Theorem}

\begin{Remark} The sequence 
$$
\frac{\dim_kL_n}{n^r}
$$
in Theorem \ref{TheoremN1} must be bounded by Theorem \ref{Theorem8}.
\end{Remark}

\begin{proof} Let $Y$ be an irreducible component of the support of $\mathcal N_X$ which has maximal dimension $r$. Let $S$ be a homogeneous coordinate ring of $X$, which we may assume is saturated, so that the natural graded homomorphism $S\rightarrow \bigoplus_{n\ge 0}\Gamma(X,\mathcal O_X(n))$ is an inclusion. Let $P_Y$ be the homogeneous prime ideal of $Y$ in $S$. There exist homogeneous elements $z_0,\ldots,z_r\in S$ such that if $\overline z_i$ is the image of $z_i$ in $S/P_Y$, then $\overline z_1,\ldots, \overline z_r$ is a homogeneous system of parameters in $S/P_Y$
(by Lemma 1.5.10 and Proposition 1.5.11 \cite{BH}). We can assume that $\deg z_0=1$ since some linear form in $S$ is not in $P_Y$, so it is not a zero divisor $S/P_Y$. We can take $z_0$ to be this form
(If $k$ is infinite, we can take all of the $z_i$ to have degree 1). 
$k[\overline z_0,\ldots,\overline z_r]$ is a weighted polynomial ring (by Theorem 1.5.17 \cite{BH}), so $A:= k[z_0,\ldots,z_r]\cong k[\overline z_0,\ldots,\overline z_d]$ is  a weighted polynomial ring. Let $N$ be the nilradical of $S$. The sheafification of $N$ is $\mathcal N_X$. $P_Y$ is a minimal prime of $N$, so there exists a homogeneous element $x\in N$ such that $\mbox{ann}_S(x)=P_Y$. $N_{P_Y}\ne 0$ in $S_{P_Y}$, so $(P_Y)_{P_Y}\ne 0$. Thus $x\in P_Y$, since otherwise $0=(xP_Y)_{P_Y}=(P_Y)_{P_Y}$. Consider the graded
$k$-subalgebra $B:=A[x]=k[z_0,\ldots,z_r,x]$ of $S$. We have that $x^2=0$. Also, $\mbox{ann}_A(x)=\mbox{ann}_S(x)\cap A=P_Y\cap A=(0)$. Suppose that $ax+b=0$ with $a,b\in A$. Then $xb=0$, whence $b=0$ and thus $a=0$ also. Hence the only relation on $B$ is $x^2=0$. Let $d_i=\deg z_i$, $e=\deg x$. Recall that $d_0=1$. Let $f=\mbox{LCM}\{d_0,\ldots,d_r,e\}$.

First assume  that $r\ge 1$. For $0\le\alpha\le r$, let $M_t^{(\alpha)}$ be the $k$-vector space of homogeneous forms of degree $t$ in the weighted variables $z_0,\ldots,z_{\alpha}$. $\mathcal O_{Z_{\alpha}}(f)$ is an ample line bundle on the weighted projective space $Z_{\alpha}=\mbox{Proj}(k[z_0,\ldots,z_{\alpha}])$ (If $U_i=\mbox{Spec}(k[z_0,\ldots,z_{\alpha}]_{(z_i)})$, then $\mathcal O_{Z_{\alpha}}|U_i=z_i^{\frac{f}{d_i}}\mathcal O_{U_i}$).
$$
\dim_kM_{nf}^{(\alpha)}=\dim_k\Gamma(Z_{\alpha},\mathcal O_{Z_{\alpha}}(nf))
$$
is thus the value of a polynomial $Q_{\alpha}(n)$ in $n$ of degree $\alpha$ for $n\gg0$. Write
$$
Q_{\alpha}(n)=c_{\alpha}n^{\alpha}+\mbox{ lower order terms}.
$$

Suppose that $s\ge 0$ (and $r\ge 1$). 
Let $L_0=k$. For $n\ge 1$, let 
$$
L_n=z_0^{nf}M_{nf}^s+xz_0^{(n-\sigma(n))f-e}M^r_{(n+\sigma(n))f}\subset  B_{2nf}\subset S_{2nf}\subset \Gamma(X,\mathcal O_X(2nf))
$$
where $\sigma(n)$ is the function of (\ref{eqsigma}).
$L_mL_n\subset L_{m+n}$ since $\sigma(j)\ge \sigma(i)$ if $j\ge i$. $L=\bigoplus_{n\ge 0}L_n$ is a graded linear series with $\kappa(L)=s$. Since $B$ has $x^2=0$ as its only relation,  we have that
$$
\begin{array}{lll}
\dim_kL_n&=& \dim_k M_{nf}^s+\dim_k M^r_{(n+\sigma(n))f}\\
&=& Q_s(n)+Q_r(n+\sigma(n)),
\end{array}
$$
so that
$$
\lim_{n\rightarrow \infty}\frac{\dim_k L_n}{n^r}= \lim_{n\rightarrow \infty} \left(c_sn^{s-r}+c_r(1+(\frac{\sigma(n)}{n}))^r\right)
$$
which does not exist, even when $n$ is constrained to lie in any arithmetic sequence, since $\lim_{n\rightarrow \infty}\frac{\sigma(n)}{n}$ has this property (as commented after (\ref{eqnr1})).

Suppose that $s=-\infty$ (and $r\ge 1$). Then define the graded linear series $L$ by 
$$
L_n=xz_0^{(n-\sigma(n))f-e}M^r_{(n+\sigma(n))f}.
$$
Then $\kappa(L)=-\infty$. We compute as above that
$$
\lim_{n\rightarrow\infty} \frac{\dim_kL_n}{n^r}=\lim_{n\rightarrow \infty} c_r(1+(\frac{\sigma(n)}{n}))^r
$$
 does not exist, even when $n$ is constrained to lie in any arithmetic sequence.

Now assume that $r=s=0$. Since $\dim \mathcal N_X=0$, we have injections for all $n$,
$$
\Gamma(X,\mathcal N_X)\cong \Gamma(X,\mathcal N_X\otimes \mathcal O_X(n))\rightarrow \Gamma(X,\mathcal O_X(n)).
$$ 
 In this case $Y$ is a closed point, so that $\dim_k\Gamma(X,\mathcal I_Y\otimes\mathcal O_X(en))$ goes to infinity as $n\rightarrow\infty$ (we assume that $d=\dim X>0$). Thus for $g\gg 0$, there exists $h\in \Gamma(X,\mathcal I_Y\otimes\mathcal O_X(eg))$ such that $h\not\in \Gamma(X,\mathcal N_X\otimes\mathcal O_X(eg))$, so $h$ is not nilpotent in $S$. $h\in P_Y$ implies $hx=0$ in $S$.  Define $L_0=k$ and for $n>0$,
 $$
 L_n =\left\{\begin{array}{ll}
 kh^n&\mbox{ if }\tau(n)=0\\
 kh^n+kxz_0^{ng-e}&\mbox{ if }\tau(n)=1,
 \end{array}\right.
 $$
 where $\tau(n)$ is the function of (\ref{eqtau}). $\tau(n)$ has the property that $\tau(n)$ is not eventually constant, even when $n$ is constrained to line in an arithmetic sequence.

$L=\bigoplus_{n\ge 0}L_n$ is a graded linear series on $X$ with $\kappa(L)=0$ such that 
$\lim_{n\rightarrow\infty}\dim_kL_n$ does not exist, even when $n$ is constrained to lie in any arithmetic sequence.

The last case is when $r=0$ and $s=-\infty$. Define $L_0=k$ and for $n>0$,
$$
L_n=\left\{\begin{array}{ll}
0&\mbox{ if }\tau(n)=0\\
kxz_0^{ng-e}&\mbox{ if }\tau(n)=1.
\end{array}
\right.
$$
Then $L=\bigoplus_{n\ge 0}L_n$ is a graded linear series on $X$ with $\kappa(L)=-\infty$ such that 
$\lim_{n\rightarrow\infty}\dim_kL_n$ does not exist, even when $n$ is constrained to lie in any arithmetic sequence.

\end{proof}

\begin{Theorem}\label{TheoremN20} Suppose that $X$ is a projective nonreduced scheme over a field $k$. Suppose that $s\in \NN\cup\{-\infty\}$ satisfies
$s\le\dim \mathcal N_X$. Then there exists a graded linear series $L$ on $X$ with $\kappa(L)=s$ and a constant $\alpha>0$ such that 
$$
\alpha n^{\dim \mathcal N_X}<\dim_k L_{nm}
$$
for all $n\gg 0$.
\end{Theorem}

\begin{proof} Let $r=\dim \mathcal N_X$. When $r\ge 1$ and $s\le r$, this is established in the construction of Theorem \ref{TheoremN1}.
When $r=0$ and $s=0$, the graded linear series $L=k[t]$ (with associated line bundle $\mathcal O_X$) has $\kappa(L)=0$ and $\dim_kL_n=1$ for all $n$, so satisfies the bound.

Suppose that $r=0$ and $s=-\infty$. Then $0\ne \Gamma(X,\mathcal N_X)$ since the support of $\mathcal N_X$ is zero dimensional. Define $L_0=k$ and $L_n=\Gamma(X,\mathcal N_X)$ for $n>0$. Then $L=\bigoplus_{n\ge 0}L_n$ is a graded linear series for $\mathcal O_X$ with $\kappa(L)=-\infty$ which satisfies the bound.

\end{proof}

\begin{Theorem}\label{TheoremN2}
Suppose that $X$ is a $d$-dimensional projective scheme  over a field $k$ with $d>0$. Let $\mathcal N_X$ be the nilradical of $X$. Let $\alpha\in  \NN$. Then the following are equivalent:
\begin{enumerate}
\item[1)] For every graded linear series $L$ on $X$ with $\alpha\le\kappa(L)$, there exists a positive integer $r$ such that 
$$
\lim_{n\rightarrow\infty}\frac{\dim_kL_{a+nr}}{n^{\kappa(L)}}
$$
exists for every positive integer $a$.
\item[2)] For every graded linear series $L$ on $X$ with $\alpha\le \kappa(L)$, there exists an arithmetic sequence $a+nr$ (for fixed $r$ and $a$ depending on $L$) such that 
$$
\lim_{n\rightarrow\infty}\frac{\dim_kL_{a+nr}}{n^{\kappa(L)}}
$$
exists.
\item[3)] The nilradical $\mathcal N_X$ of $X$ satisfies $\dim \mathcal N_X<\alpha$.
\end{enumerate}
\end{Theorem}

\begin{proof} 1) implies 2) is immediate. 2) implies 3) follows from Theorem \ref{TheoremN1}. 3) implies 1) follows from Theorem \ref{Theorem8}.
\end{proof}

\section{Nonreduced zero dimensional schemes}\label{SecZero}

The case when $d=\dim X=0$ is  rather special. In fact, the implication 2) implies ) of Theorem \ref{TheoremN2} does not hold if $d=0$, as follows from
Proposition \ref{TheoremN3} below. There is however a very precise statement about what does happen in zero dimensional schemes, as we show below.

\begin{Proposition}\label{TheoremN3} Suppose that $X$ is a 0-dimensional irreducible but nonreduced $k$-scheme and $L$ is a graded linear series on $X$ with $\kappa(L)=0$. Then there exists a positive integer $r$ such that  
$$
\lim_{n\rightarrow\infty}\dim_kL_{a+nr}
$$
exists for every positive integer $a$.
\end{Proposition}

\begin{proof} With our assumptions, $X=\mbox{Spec}(A)$ where $A$ is a nonreduced Artin local ring, with $\dim_kL<\infty$, and $L$ is a graded $k$-subalgebra of $\Gamma(X,\mathcal O_X)[t]=A[t]$. The condition $\kappa(L)=0$ is equivalent to the statement that there exists $r>0$ such that $L_r$ contains a unit $u$ of $A$. We then have that 
$$
\dim_kL_{m+r}\ge \dim_kL_mL_r\ge \dim_kL_m
$$
for all $m$. Thus for fixed $a$, $\dim_kL_{a+nr}$ must stabilize for large $n$.
\end{proof}

We do not have such good behavior for graded linear series $L$ with $\kappa(L)=\infty$. 

\begin{Proposition}\label{TheoremN6} Suppose that $X$ is a 0-dimensional  nonreduced $k$-scheme. Then there exists a  graded linear series $L$ on $X$ with $\kappa(L)=-\infty$, such that 
$$
\lim_{n\rightarrow\infty}\dim_kL_{n}
$$
does not exist, even when $n$ is constrained to lie in any arithmetic sequence.
\end{Proposition}

\begin{proof}  $X=\mbox{Spec}(A)$ where $A=\bigoplus_{i=1}^sA_i$, with $s$ the number of irreducible components of $X$ and the $A_i$ are  Artin local rings with $\dim_kA<\infty$. Let $m_{A_1}$ be the maximal ideal of $A_1$.  
There exists a number $0<t$ such that $m_{A_1}^t\ne 0$ but $m_{A_1}^{t+1}=0$. Let $\tau(n)$ be the function defined in (\ref{eqtau}).

Define a graded linear series $L^1$  on $\mbox{Spec}(A_1)$ by $L^1_n=m{A_1}^{t+\tau(n)}$. Then $\lim_{n\rightarrow \infty}\dim_k L_n^1$ does not exist, even when $n$ is constrained to  lie in an arithmetic sequence. Extend $L$ to a graded linear series $L$ on $X$ with $\kappa(L)=-\infty$ be setting
$L_n=L_n^1\bigoplus (0)\bigoplus\cdots\bigoplus (0)$.
\end{proof}

It follows that the conclusions of Proposition \ref{TheoremN3} do not hold in nonreduced 0-dimensional schemes which are not irreducible.

\begin{Proposition}\label{TheoremN7}  Suppose that $X$ is a 0-dimensional {\it nonirreducible} and nonreduced $k$-scheme. Then there exists a  graded linear series $L$ on $X$ with $\kappa(L)=0$, such that 
\begin{equation}\label{eqN10}
\lim_{n\rightarrow\infty}\dim_kL_{n}
\end{equation}
does not exist, even when $n$ is constrained in any arithmetic sequence.
\end{Proposition}

\begin{proof} $X=\mbox{Spec}(A)$ where $A=A_1\bigoplus A_2$, with $A_1$ an Artin local ring and $A_2$  an Artin ring.
 A graded linear series $L$ on $X$ is a graded $k$-subalgebra of $A[t]$. Let $L^2$ be a graded linear series on $\mbox{Spec}(A_2)$ with $\kappa(L^2)=-\infty$, such that the conclusions of Proposition \ref{TheoremN6} hold. Then the linear series $L$ on $X$ defined by 
$$
L_n=A_1\bigoplus (L^2)_n
$$
has $\kappa(L)=0$, but 
$$
\lim_{n\rightarrow\infty}\dim_kL_{n}=1+\lim_{n\rightarrow \infty}\dim_k L^2_n
$$
does not exist, even when $n$ is constrained to lie in any arithmetic sequence.  

\end{proof}

In particular, the conclusions of Theorem \ref{TheoremN2} are true for 0-dimensional projective $k$-schemes which are not irreducible.

\section{Applications to asymptotic multiplicities}\label{SecApp}

In this section, we apply Theorem \ref{Theorem1} and its method of proof, to generalize some of the applications in \cite{C1} to analytically irreducible and analytically unramified local rings.

\begin{Theorem}\label{Theorem14} Suppose that $R$ is an analytically unramified  local ring of dimension $d>0$. 
Suppose that $\{I_i\}$ and $\{J_i\}$ are graded families of ideals in $R$. Further suppose that $I_i\subset J_i$ for all $i$ and there exists $c\in\ZZ_+$ such that
\begin{equation}\label{eq60}
m_R^{ci}\cap I_i= m_R^{ci}\cap J_i
\end{equation}
 for all $i$. Then the limit
$$
\lim_{i\rightarrow \infty} \frac{\ell_R(J_i/I_i)}{i^{d}}
$$
exists.
\end{Theorem}

Theorem \ref{Theorem14}  is proven for local rings $R$ which are regular, or normal excellent and equicharacteristic in \cite{C1}.

\begin{Remark} A reduced analytic local ring $R$ satisfies the hypotheses  of Theorem \ref{Theorem14}. 
An analytic local ring is excellent by Theorem 102 on page 291 \cite{Ma}. A reduced, excellent local ring is unramified by
(x) of Scholie 7.8.3 \cite{EGAIV}.
\end{Remark}

\begin{proof}  We may assume that $R$ is complete, by replacing $R$, $I_i$, $J_i$ by $\hat R$, $I_i\hat R$, $J_i\hat R$. 

First suppose that $R$ is analytically irreducible. We will prove the theorem in this case. We will apply the method of Theorem \ref{Theorem1}.
 Construct the regular local ring $S$ by the argument of the proof of Theorem \ref{Theorem1}.

Let $\nu$ be the valuation of $Q(R)$ constructed from $S$ in the proof of Theorem \ref{Theorem1}, with associated valuation ideals $K_{\lambda}$ in the valuation ring $V_{\nu}$ of $\nu$. Let $k=R/m_R$ and $k'=S/m_S=V_{\nu}/m_{\nu}$.

By (\ref{eqred50}), there exists
$\alpha\in \ZZ_+$ such that 
$$
K_{\alpha n}\cap R\subset m_R^n
$$
for all $n\in \ZZ_+$. 
We have that 
$$
K_{\alpha cn}\cap I_n=K_{\alpha cn}\cap J_n
$$
for all $n$. Thus
\begin{equation}\label{eq31}
\ell_R(J_n/I_n)=\ell_R(J_n/K_{\alpha cn}\cap J_n)-\ell_R(I_n/K_{\alpha cn}\cap I_n)
\end{equation}
for all $n$. Let $\beta=\alpha c$. For $t\ge 1$, let
$$
\Gamma(J_*)^{(t)}=\{\begin{array}{l}
(n_1,\ldots,n_d,i)\mid \dim_k J_i\cap K_{n_1\lambda_1+\cdots+n_d\lambda_d}/J_i\cap K_{n_1\lambda_1+\cdots+n_d\lambda_d}^+\ge t\\
\mbox{ and }n_1+\cdots+n_d\le \beta i, 
\end{array}\}
$$
and
$$
\Gamma(I_*)^{(t)}=\{\begin{array}{l}
(n_1,\ldots,n_d,i)\mid \dim_k I_i\cap K_{n_1\lambda_1+\cdots+ n_d\lambda_d}/I_i\cap K_{n_1\lambda_1+\cdots+n_d\lambda_d}^+\ge t\\
\mbox{ and }n_1+\cdots+n_d\le \beta i
\end{array}\}.
$$
We have that 
\begin{equation}\label{eq32}
\ell_R(J_n/I_n)=(\sum_{t=1}^{[k':k]} \# \Gamma(J_*)^{(t)}_n)-(\sum_{t=1}^{[k':k]} \# \Gamma(I_*)^{(t)}_n)
\end{equation}
as explained in the proof of Theorem \ref{Theorem1}. As in the proof of Lemma \ref{Lemmared3}, we have that $\Gamma(J_*)^{(t)}$ and $\Gamma(I_*)^{(t)}$
  satisfy the conditions
(\ref{Cone2}) and (\ref{Cone3})of Theorem \ref{ConeTheorem1}. Thus
$$
\lim_{n\rightarrow \infty} \frac{\# \Gamma(J_*)^{(t)}_n}{n^d}={\rm vol}(\Delta(\Gamma(J_*))\mbox{ and }\lim_{n\rightarrow \infty} \frac{\# \Gamma(I_*)^{(t)}_n}{n^d}
={\rm vol}(\Delta(\Gamma(I_*)^{(t)})
$$
by Theorem \ref{ConeTheorem1}. 
The theorem, in the case when $R$ is analytically irreducible, now follows from (\ref{eq32}).

Now suppose that $R$ is only analytically unramified. We may continue to assume that $R$ is  complete. 
Let $P_1,\ldots,P_s$ be the minimal primes of $R$. Let $R_i=R/P_i$ for $1\le i\le s$. Let $T=\bigoplus_{i=1}^s R_i$ and $\phi:R\rightarrow T$ be the natural inclusion. By Artin-Rees, there exists a positive integer $\lambda$ such that
$$
\omega_n:=\phi^{-1}(m_R^nT)=R\cap m_R^nT\subset m_R^{n-\lambda}
$$
for all $n\ge \lambda$. 
Thus
$$
m_R^n\subset \omega_n\subset m_R^{n-\lambda}
$$
for all $n$.
We have that
$$
\omega_n=\phi^{-1}(m_R^nT)=\phi^{-1}(m_R^nR_1\bigoplus\cdots\bigoplus m_R^nR_s)
=(m_R^n+P_1)\cap \cdots \cap(m_R^n+P_s).
$$
Let $\beta=(\lambda+1)c$.
Now $\omega_{\beta n}\subset m_R^{c(\lambda+1)n-\lambda}\subset m_R^{cn}$ for all $n\ge 1$, so that 
$$
\omega_{\beta n}\cap I_n=\omega_{\beta n}\cap(m_R^{cn}\cap I_n)=\omega_{\beta n}\cap(m_R^{cn}\cap J_n)=\omega_{\beta n}\cap J_n
$$
for all $n\ge 1$, so 
\begin{equation}\label{eqNew1}
\ell_R(J_n/I_n)=\ell_R(J_n/\omega_{\beta n}\cap J_n)-\ell_R(I_n/\omega_{\beta n}\cap I_n)
\end{equation}
for all $n\ge 1$.

Define $L_0^j=R$ for $0\le j\le s$, and for $n>0$, define $L_n^0=J_n$ and
$$
L_n^j=(m_R^{\beta n}+P_1)\cap\cdots\cap (m_R^{\beta n}+P_j)\cap J_n
$$
for $1\le j\le s$. For fixed $j$, with $0\le j\le s$, $\{L_n^j\}$ is a graded family of ideals in $R$. For $n\ge 1$, we have isomorphisms
$$
L_n^j/L_n^{j+1}= L_n^j/(m_R^{\beta n}+P_{j+1})\cap L_n^j\cong L_n^jR_{j+1}/(L_n^jR_{j+1})\cap m_{R_{j+1}}^{\beta n}
$$
for $0\le j\le s-1$, and 
$$
L_n^s=\omega_{\beta n}\cap J_n.
$$
Thus
\begin{equation}\label{eqNew2}
\ell_R(J_n/\omega_{\beta n}\cap J_n)=\sum_{j=0}^{s-1}\ell_R(L_n^j/L_n^{j+1})
=\sum_{j=0}^{s-1}\ell_{R_{j+1}}\left(L_n^jR_{j+1}/(L_n^jR_{j+1})\cap m_{R_{j+1}}^{\beta n}\right).
\end{equation}

For some fixed $j$ with $0\le j\le s-1$, let $\overline R=R_{j+1}$, $\overline J_n=L_n^j\overline R$  and $\overline I_n =\overline J_n\cap m_{\overline R}^{\beta n}$.
$\{\overline I_n\}$ and $\{\overline J_n\}$ are graded families of ideals in  $\overline R$
and $m_{\overline R}^{\beta n}\cap\overline I_n=m_{\overline R}^{\beta n}\cap\overline J_n$ for all $n$.

Since $\dim \overline R\le \dim R=d$ and $\overline R$ is  analytically irreducible, by the first part of the proof we have that
$$
\lim_{n\rightarrow \infty}\frac{\ell_{\overline R}(\overline J_n/\overline I_n)}{n^d}
$$
exists, and from (\ref{eqNew2}), we have that
$$
\lim_{n\rightarrow \infty}\frac{\ell_R(J_n/\omega_{\beta n}\cap J_n)}{n^d}
$$
exists. The same argument applied to the graded family of ideals $\{I_n\}$ in $R$ implies that 
$$
\lim_{n\rightarrow \infty}\frac{\ell_R(I_n/\omega_{\beta n}\cap I_n)}{n^d}
$$
exists. Finally, (\ref{eqNew1}) implies that the limit
$$
\lim_{n\rightarrow \infty}\frac{\ell_R(J_n/I_n)}{n^d}
$$
exists.

\end{proof}

If $R$ is a local ring and $I$ is an ideal in $R$ then the saturation of $I$ is 
$$
I^{\rm sat}=I:m_R^{\infty}=\cup_{k=1}^{\infty}I:m_R^k.
$$

\begin{Corollary}\label{Corollary5} Suppose that $R$ is an analytically unramified local ring of dimension $d>0$ and  $I$ is an ideal in $R$. Then the limit
$$
\lim_{i\rightarrow \infty} \frac{\ell_R((I^i)^{\rm sat}/I^i)}{i^{d}}
$$
exists.

\end{Corollary}

Since $(I^n)^{\rm sat}/I^n\cong H^0_{m_R}(R/I^n)$, the epsilon multiplicity of Ulrich and Validashti \cite{UV}
$$
\epsilon(I)=\limsup \frac{\ell_R(H^0_{m_R}(R/I^n))}{n^d/d!}
$$
exists as a limit, under the assumptions of Corollary \ref{Corollary5}.

Corollary \ref{Corollary5} is proven for more general families of modules when $R$ is a local domain which is essentially of finite type over a perfect field $k$ such that $R/m_R$ is algebraic over $k$ in \cite{C}. The corollary is proven with more restrictions on $R$ in Corollary 6.3 \cite{C1}. The limit in corollary \ref{Corollary5} can be irrational, as shown in \cite{CHST}.

\begin{proof} By Theorem 3.4 \cite{S}, there exists $c\in\ZZ_+$ such that each power $I^n$ of $I$ has an irredundant primary decomposition 
$$
I^n=q_1(n)\cap\cdots\cap q_s(n)
$$
where $q_1(n)$ is $m_R$-primary and $m_R^{nc}\subset q_1(n)$ for all $n$. Since $(I^n)^{\rm sat}=q_2(n)\cap \cdots\cap q_s(n)$,
we have that 
$$
I^n\cap m_R^{nc}=m_R^{nc}\cap q_2(n)\cap \cdots \cap q_s(n)=m_R^{nc}\cap (I^n)^{\rm sat}
$$
for all $n\in \ZZ_+$. Thus the corollary follows from Theorem \ref{Theorem14}, taking $I_i=I^i$ and $J_i=(I^i)^{\rm sat}$.

\end{proof}

A stronger version of the previous corollary is true. 
The following corollary proves a formula proposed by Herzog, Puthenpurakal and Verma in the introduction to \cite{HPV}.
The formula is proven with more restrictions on $R$ in Corollary 6.4 \cite{C1}.

Suppose that $R$ is a ring, and $I,J$ are ideals in $R$. Then the $n^{\rm th}$ symbolic power of $I$ with respect to $J$ is
$$
I_n(J)=I^n:J^{\infty}=\cup_{i=1}^{\infty}I^n:J^i.
$$

\begin{Corollary}\label{Cor5} Suppose that $R$ is an analytically unramified local ring of dimension $d$. 
 Suppose that $I$ and $J$ are ideals in $R$.  
 Let $s$ be the constant  limit dimension of $I_n(J)/I^n$ for $n\gg 0$. Suppose that $s<d$. Then
$$
\lim_{n\rightarrow \infty} \frac{e_{m_R}(I_n(J)/I^n)}{n^{d-s}}
$$
exists.
\end{Corollary}

\begin{proof} There exists a positive integer $n_0$ such that the  set of associated primes of $R/I^n$ stabilizes for
$n\ge n_0$ by  \cite{Br}. Let $\{p_1,\ldots, p_t\}$ be this set of associated primes.  We thus  have irredundant primary decompositions
for $n\ge n_0$,
\begin{equation}\label{eq**}
I^n=q_1(n)\cap \cdots \cap q_t(n),
\end{equation}
where $q_i(n)$ are $p_i$-primary.

We further have that 
\begin{equation}\label{eq*}
I^n:J^{\infty}=\cap_{J\not\subset p_i}q_i(n).
\end{equation}
Thus $\dim I_n(J)/I^n$ is constant for $n\ge n_0$. Let $s$ be this limit dimension. The set 
$$
A=\{p\in \cup_{n\ge n_0}{\rm Ass}(I_n(J)/I^n)\mid n\ge n_0\mbox{ and }\dim R/p=s\}
$$
is a finite set. Moreover, every such prime is in ${\rm Ass}(I_n(J)/I^n$ for all $n\ge n_0$. For $n\ge n_0$, we have 
by the additivity formula (V-2 \cite{Se} or Corollary 4.6.8, page 189 \cite{BH}), that
$$
e_{m_R}(I_n(J)/I^n)=\sum_{p}\ell_{R_p}((I_n(J)/I^n)_p)e(m_{R/p})
$$
where the sum is over the finite set of primes $p\in \mbox{Spec}(R)$ such that $\dim R/p=s$. This sum is thus over the finite set $A$.

Suppose that $p\in A$ and $n\ge n_0$. Then 
$$
I^n_p=\cap q_i(n)_p
$$
where the intersection is over the $q_i(n)$ such that $p_i\subset p$, and
$$
I_n(J)=\cap q_i(n)_p
$$
where the intersection is over the $q_i(n)$ such that $J\not\subset p_i$ and $p_i\subset p$. Thus  there exists an index $i_0$ such that $p_{i_0}=p$ and 
$$
I^n_p=q_{i_0}(n)_p\cap I_n(J)_p.
$$
By (\ref{eq**}), 
$$
(I^n_p)^{\rm sat}=I_n(J)_p
$$
for $n\ge n_0$. Since $R_p$ is analytically unramified (by \cite{R3} or Proposition 9.1.4 \cite{SH}) and $\dim R_p\le d-s$, by Corollary \ref{Corollary5}, the limit
$$
\lim_{n\rightarrow \infty} \frac{\ell_R((I_n(J)/I_n)_p)}{n^{d-s}}
$$
exists.
\end{proof}

\vskip .2truein
We now establish some Volume = Multiplicity formulas.

\begin{Theorem}\label{Theorem15} Suppose that $R$ is a $d$-dimensional analytically unramified local ring and  $\{I_i\}$ is a graded family of $m_R$-primary ideals in $R$. Then 
$$
\lim_{n\rightarrow \infty}\frac{\ell_R(R/I_n)}{n^d/d!}=\lim_{p\rightarrow \infty}\frac{e(I_p)}{p^d}
$$
exists.
Here $e(I_p)$ is the multiplicity
$$
e(I_p)=e_{I_p}(R)=\lim_{k\rightarrow \infty} \frac{\ell_R(R/I_p^k)}{k^d/d!}.
$$
\end{Theorem}

Theorem \ref{Theorem15} is proven for valuation ideals associated to an Abhyankar valuation in a regular local ring which is essentially of finite type over a field in  \cite{ELS}, for general families of $m_R$-primary ideals when $R$ is a regular local ring containing a field in \cite{Mus} and when $R$ is a local domain which is essentially of finite type over an algebraically closed field $k$ with $R/m_R=k$ in Theorem 3.8 \cite{LM}. 
It is proven when $R$ is regular or $R$ is analytically unramified with perfect residue field in Theorem 6.5 \cite{C1}.

\begin{proof} 
There exists $c\in\ZZ_+$ such that
$m_R^{c}\subset I_1$. 

We first prove the theorem when $R$ is analytically irreducible, and so satisfies the assumptions of Theorem \ref{Theorem1}. 
We may assume that $R$ is complete.
Let $\nu$ be the valuation of $Q(R)$ constructed from $S$ in the proof of Theorem \ref{Theorem1}, with associated valuation ideals $K_{\lambda}$ in the valuation ring $V_{\nu}$ of $\nu$. Let $k=R/m_R$ and $k'=S/m_S=V_{\nu}/m_{\nu}$.

Apply (\ref{eqred50}) to find
$\alpha\in \ZZ_+$ such that 
$$
K_{\alpha n}\cap R\subset m_R^n
$$
for all $n\in \NN$. 
We have that 
$$
K_{\alpha c n}\cap R \subset m_R^{cn}\subset I_n
$$
 for all $n$.

For $t\ge 1$, let
$$
\Gamma(I_*)^{(t)}=\{\begin{array}{l}
(n_1,\ldots,n_d,i)\mid \dim_k I_i\cap K_{n_1\lambda_1+\cdots+n_d\lambda_d}/I_i\cap K_{n_1\lambda_1+\cdots+n_d\lambda_d}^+\ge t\\
\mbox{ and }n_1+\cdots+n_d\le\alpha c i
\end{array} \},
$$
and
$$
\Gamma(R)^{(t)}=\{\begin{array}{l}
(n_1,\ldots,n_d,i)\mid \dim_k R\cap K_{n_1\lambda_1+\cdots+n_d\lambda_d}/R\cap K_{n_1\lambda-1+\cdots+n_d\lambda_d}^+\ge t\\
 \mbox{ and }n_1+\cdots+n_d\le\alpha c i
 \end{array}\}.
$$
As in the proofs of Lemmas \ref{Lemmared1} and \ref{Lemmared3},  $\Gamma(I_*)^{(t)}$ and $\Gamma(R)^{(t)}$ satisfy the conditions (\ref{Cone2}) and (\ref{Cone3})
of Theorem \ref{ConeTheorem1} when they are not contained in $\{0\}$.
For fixed $p\in \ZZ_+$ and $t\ge 1$, let
$$
\Gamma(I_*)(p)^{(t)}=\{\begin{array}{l}
(n_1,\ldots,n_d,kp)\mid \dim_k I_p^k\cap K_{n_1\lambda_1+\cdots+n_d\lambda_d}/I_p^k\cap K_{n_1\lambda_1+\cdots+n_d\lambda_d}^+\ge t\\
 \mbox{ and }n_1+\cdots+n_d\le\alpha c kp
 \end{array}\}.
$$
We have inclusions of semigroups
$$
k*\Gamma(I_*)_p^{(t)}\subset \Gamma(I_*)(p)^{(t)}_{kp}\subset \Gamma(I_*)^{(t)}_{kp}
$$
for all $p$, $t$ and $k$.

By Theorem \ref{ConeTheorem2}, given $\epsilon>0$, there exists $p_0$ such that $p\ge p_0$ implies
$$
{\rm vol}(\Delta(\Gamma(I_*)^{(t)})-\frac{\epsilon}{[k':k]}\le \lim_{k\rightarrow \infty}\frac{\#(k*\Gamma(I_*)^{(t)}_p)}{k^dp^d}.
$$
Thus
$$
{\rm vol}(\Delta(\Gamma(I_*)^{(t)})-\frac{\epsilon}{[k':k]}\le \lim_{k\rightarrow\infty}\frac{\#\Gamma(I_*)(p)^{(t)}_{kp}}{k^dp^d}
\le{\rm vol}(\Delta(\Gamma(I_*)^{(t)}).
$$
Again by Theorem \ref{ConeTheorem2}, we can choose $p_0$ sufficiently large that we also have that
$$
{\rm vol}(\Delta(\Gamma(R)^{(t)})-\frac{\epsilon}{[k':k]}\le\lim_{k\rightarrow \infty}\frac{\#\Gamma(R)^{(t)}_{kp}}{k^dp^d}\le{\rm vol}(\Delta(\Gamma(R)^{(t)})).
$$
Now 
$$
\ell_R(R/I_p^k)=(\sum_{t=1}^{[k':k]}\#\Gamma(R)^{(t)}_{kp})-(\sum_{t=1}^{[k':k]} \#\Gamma(I_*)(p)^{(t)}_{kp})
$$
and 
$$
\ell_R(R/I_n)=(\sum_{t=1}^{[k':k]} \#\Gamma(R)^{(t)}_n)-(\sum_{t=1}^{[k':k]} \#\Gamma(I_*)^{(t)}_n).
$$
By Theorem \ref{ConeTheorem1},
$$
\lim_{n\rightarrow \infty}\frac{\ell_R(R/I_n)}{n^d}=(\sum_{t=1}^{[k':k]}{\rm vol}(\Delta(\Gamma(R)^{(t)}))-(\sum_{t=1}^{[k':k]}{\rm vol}(\Delta(\Gamma(I_*)^{(t)}))).
$$
Thus
$$
\lim_{n\rightarrow\infty}\frac{\ell_R(R/I_n)}{n^d}-\epsilon\le \lim_{k\rightarrow\infty}\frac{\ell_R(R/I_p^k)}{k^dp^d}
=\frac{e(I_p)}{d!p^d}\le \lim_{n\rightarrow \infty}\frac{\ell_R(R/I_n)}{n^d}+\epsilon.
$$
Taking the limit as $p\rightarrow \infty$, we obtain the conclusions of the theorem.

Now assume that $R$ is analytically unramified. We may assume that $R$ is complete and reduced 
since 
$$
\ell_R(R/I_p^k)=\ell_{\hat R}(\hat R/I_p^k\hat R)
$$
for all $p,k$.

Suppose that the minimal primes of (the reduced ring) $R$ are $\{q_1,\ldots, q_s\}$. Let  $R_i=R/q_i$. $R_i$ are complete local domains.
We  have that
$$
\frac{e_d(I_p,R)}{p^d}=\sum_{i=1}^s\frac{e_d(I_pR_i,R_i)}{p^d}
$$
by the additivity  formula (page V-3 \cite{Se} or Corollary 4.6.8, page 189 \cite{BH}) or directly from Lemma \ref{Lemma5}. We also have  that
$$
\lim_{n\rightarrow\infty}\frac{\ell_R(R/I_n)}{n^d}=\sum_{i=1}^s\lim_{n\rightarrow\infty}\frac{\ell_R(R_i/I_nR_i)}{n^d}
$$
by Lemma \ref{Lemma5}. Since each $R_i$ is analytically irreducible, the limits 
$$
\lim_{n\rightarrow \infty}\frac{\ell_R(R_i/I_nR_i)}{n^d}
$$
 exist by Theorem \ref{Theorem1}.

\end{proof}

Suppose that $R$ is a Noetherian ring, and $\{I_i\}$ is a graded  family of ideals in $R$.
Let
$$
s=s(I_*)=\limsup \dim R/ I_i.
$$
Let $i_0\in \ZZ_+$ be the smallest integer such that
\begin{equation}\label{eq40}
\mbox{$\dim R/I_i \le  s$ for $i\ge i_0$.}
\end{equation}
For $i\ge i_0$ and $p$ a prime ideal in $R$ such that $\dim R/p=s$, we have that $(I_i)_p=R_p$ or $(I_i)_p$ is $p_p$-primary.

$s$ is in general not a limit, as is shown by Example 6.6 \cite{C1}.

Let
$$
T=T(I_*)=\{p\in \mbox{spec}(R)\mid \dim R/p=s\mbox{ and there exist arbitrarily large $j$ such that $(I_j)_p\ne R_p$}\}.
$$

We recall some lemmas from \cite{C1}.

\begin{Lemma}\label{Lemma10}(Lemma 6.7 \cite{C1})  $T(I_*)$ is a finite set.
\end{Lemma}

\begin{Lemma}\label{Lemma11}(Lemma 6.8 \cite{C1}) There exist $c=c(I_*)\in \ZZ_+$ such that if $j\ge i_0$ and $p\in T(I_*)$, then 
$$
p^{jc}R_p\subset I_jR_p.
$$
\end{Lemma}

Let 
$$
A(I_*)=\{q\in T(I_*)\mid  \mbox{$I_nR_q$ is $q_q$-primary for  $n\ge i_0$}\}.
$$

\begin{Lemma}\label{Lemma50}(Lemma 6.9 \cite{C1})  Suppose that $q\in T(I_*)\setminus A(I_*)$. Then there exists $b\in \ZZ_+$ such that $q_q^b\subset (I_n)_q$ for all $n\ge i_0$.
\end{Lemma}

We obtain the following asymptotic additivity formula. It is proven in Theorem 6.10 \cite{C1}, with the additional assumption that $R$ is regular or analytically unramified of equicharacteristic zero.

\begin{Theorem}\label{Theorem13} Suppose that $R$ is a $d$-dimensional analytically unramified local ring 
and $\{I_i\}$ is a graded family of  ideals in $R$. Let
$s=s(I_*)=\limsup \dim R/I_i$ and $A=A(I_*)$. Suppose that $s<d$. Then 
$$
\lim_{n\rightarrow \infty}\frac{e_{s}(m_R,R/I_n)}{n^{d-s}/(d-s)!}=\sum_{q\in A}\left(\lim_{k\rightarrow \infty}\frac{e((I_k)_q)}{k^{d-s}}\right) e(m_{R/q}).
$$
\end{Theorem}

\begin{proof} 
 Let $i_0$ be the (smallest) constant satisfying (\ref{eq40}). By the additivity formula (V-2 \cite{Se} or Corollary 4.6.8, page 189 \cite{BH}), for $i\ge i_0$,
$$
e_{s}(m_R,R/I_i)=\sum_p\ell_{R_p}(R_p/(I_i)_p)e_{m_R}(R/p)
$$
where the sum is over all prime ideals $p$ of $R$ with $\dim R/p=s$. By Lemma \ref{Lemma10}, for $i\ge i_0$,
the sum is actually over the finite set $T(I_*)$ of prime ideals of $R$.

For $p \in T(I_*)$, $R_p$ is a local ring of dimension $\le d-s$. Further, $R_p$  is analytically unramified
(by \cite{R3} or Prop 9.1.4 \cite{SH}). By Lemma \ref{Lemma11}, and by  Theorem \ref{Theorem2},  replacing $(I_i)_p$ with $p_p^{ic}$ if $i<i_0$, we have that
$$
\lim_{i\rightarrow \infty}\frac{\ell_{R_p}(R_p/(I_i)_p)}{i^{d-s}}
$$
exists. Further, this limit is zero if $p\in T(I_*)\setminus A(I_*)$ by Lemma \ref{Lemma50}, and since $s<d$. Finally, we have
$$
\lim_{i\rightarrow \infty}\frac{\ell_{R_q}(R_q/(I_i)_q)}{i^{d-s}/(d-s)!}=\lim_{k\rightarrow \infty}\frac{e_{(I_k)_q}(R_q)}{k^{d-s}}
$$
for $q\in A(I_*)$ by Theorem \ref{Theorem15}.

\end{proof}


\begin{thebibliography}{1000000000}
\bibitem{Br} M. Brodmann, Asymptotic stablity of $\mbox{Ass}(M/I^nM)$, Proc. Amer. Math. Soc 74 (1979), 16 - 18.
\bibitem{BH} W. Bruns and J. Herzog, Cohen-Macaulay rings, Cambridge University Press, 1993.
\bibitem{C1} S.D. Cutkosky, Multiplicities associated to graded families of ideals, to appear in Algebra and Number Theory, arXiv:1212.6186.
\bibitem{C2} S.D. Cutkosky, The asymptotic  growth of graded linear series on arbitrary projective schemes, arXiv:1206.4077.
\bibitem{C} S.D. Cutkosky, Asymptotic growth of saturated powers and epsilon multiplicity, Math. Res. Lett. 18 (2011), 93 - 106,
\bibitem{CDK} S.D. Cutkosky, K. Dalili and O. Kashcheyeva, Growth of rank 1 valuations, Comm. Algebra 38 (2010), 2768 - 2789.
\bibitem{CHST} S.D. Cutkosky, T. Ha, H. Srinivasan and E. Theodorescu, Asymptotic behavior of length of local cohomology,
Canad. J. Math. 57 (2005), 1178 -1192.
\bibitem{CS} S.D. Cutkosky and V. Srinivas, On a problem of Zariski on dimensions of linear systems, Ann. Math. 137 (1993), 531 - 559.
\bibitem{DJ} A. J. de Jong, Smoothness, semi-stability and alterations, Publications Mathematiques I.H.E.S. 83 (1996), 51 - 93.

\bibitem{ELS} L. Ein, R. Lazarsfeld and K. Smith, Uniform Approximation of Abhyankar valuation ideals in smooth function fields,
Amer. J. Math. 125 (2003), 409 - 440. 
\bibitem{Fuj} T. Fujita, Approximating Zariski decomposition of big line bundles, Kodai Math. J. 17 (1994), 1-3.
\bibitem{EGAIV} A. Grothendieck, and A. Dieudonn\'e, El\'ements de g\'eom\'etrie alg\'ebrique IV, vol. 2, Publ. Math. IHES 24 (1965).
\bibitem{H} R. Hartshorne, Algebraic Geometry, Springer Verlag 1977.

\bibitem{HPV} J. Herzog, T. Puthenpurakal, J. Verma, Hilbert polynomials and powers of ideals, Math. Proc. Cambridge Math. Soc. 145 (2008), 623 - 642.
\bibitem{I} S. Iitaka, Algebraic Geometry, Springer Verlag, 1982.
\bibitem{Kh} A. Khovanskii, Newton polyhedron, Hilbert polynomial, and sums of finite sets, Funct. Anal, Appl. 26 (1993), 276 - 281.
\bibitem{KK} K. Kaveh and G. Khovanskii, Newton-Okounkov bodies, semigroups of integral points, graded algebras and intersection theory, to appear in Annals of Math., arXiv:0904.3350v3.


\bibitem{La} R. Lazarsfeld, Positivity in Algebraic Geometry, I and II, Ergebnisse der Mathematik und ihrer Grenzgebiete, vols 48 and 49, Springer Verlag, Berlin 2004.
\bibitem{LM} R. Lazarsfeld and M. Mustata, Convex bodies associated to linear series, Ann. Sci. Ec. Norm. Super 42 (2009) 783 - 835.
\bibitem{Ma} H. Matsumura, Commutative Algebra, 2nd edition, Benjamin/Cummings (1980).
\bibitem{Mus} M. Mustata, On multiplicities of graded sequence of ideals, J. Algebra 256 (2002), 229-249.
\bibitem{N} M. Nagata, Local Rings, Wiley, 1962.
\bibitem{Ok} A. Okounkov, Why would multiplicities be log-concave?, in The orbit method in geometry and physics, Progr. Math. 213, 2003, 329-347.
\bibitem{R3}  D. Rees, A note on analytically unramified local rings, J. London Math. Soc. 36 (1961), 24 - 28.
\bibitem{Se} J.P. Serre, Alg\`ebre locale. Multiplicit\'es. LNM 11, Springer, 1965.
\bibitem{S} I. Swanson, Powers of ideals: primary decomposition, Artin-Rees lemma and regularity,
Math. Annalen 307 (1997), 299 - 313.
\bibitem{SH} I. Swanson and C. Huneke, Integral closure of ideals, rings and modules, Cambridge Univ. Press, 200.
\bibitem{T} S. Takagi, Fujita's approximation theorem in positive characteristics, J. Math. Kyoto Univ, 47 (2007), 179 - 202.
\bibitem{UV} B. Ulrich and J. Validashti, Numerical criteria for integral dependence,  Math. Proc. Camb. Phil. Soc. 151 (2011), 95 - 102.
\bibitem{ZS2} O. Zariski and P. Samuel, Commutative Algebra Vol II, Van Nostrand (1960).
\bibitem{Z} O. Zariski, The concept of a simple point on an abstract algebraic variety, Trans. Amer. Math. Soc. 62 (1947), 1- 52.
\end{thebibliography}
\end{document}